\documentclass[11pt]{amsart}

\IfFileExists{srcltx.sty}{\usepackage{srcltx}}

\numberwithin{equation}{section}

\usepackage{tikz}
\usetikzlibrary{calc,decorations.pathmorphing}
\usetikzlibrary{arrows}

\usepackage[latin1]{inputenc}
\usepackage{amssymb}
\usepackage{xspace,amsfonts,euscript} 
\usepackage{amsthm,amsmath}
\usepackage{palatino}
\usepackage{euscript}
\input xy \xyoption {all}

\usepackage{pdfsync}

\RequirePackage{color}
\definecolor{myred}{rgb}{0.75,0,0}
\definecolor{mygreen}{rgb}{0,0.5,0}
\definecolor{myblue}{rgb}{0,0,0.65}

\RequirePackage{ifpdf}
\ifpdf
  \IfFileExists{pdfsync.sty}{\RequirePackage{pdfsync}}{}
  \RequirePackage[pdftex,
   colorlinks = true,
   urlcolor = myblue, 
   citecolor = mygreen, 
   linkcolor = myred, 
   pagebackref,
   bookmarksopen=true]{hyperref}
\else
  \RequirePackage[hypertex]{hyperref}
\fi

\usepackage{url}

\RequirePackage{ae, aecompl, aeguill} 

  \def\bg{{\mathfrak b}}  
    \def\CM{{\mathbb{C}}}
    \def\DM{{\mathbb{D}}}

  \def\hg{{\mathfrak h}}

    \def\PM{{\mathbb{P}}}
    
    \def\RM{{\mathbb{R}}}
    
  \def\tg{{\mathfrak t}}

    \def\ZM{{\mathbb{Z}}}

    \def\BC{{\mathcal{B}}}

    \def\OC{{\mathcal{O}}}

\def\BS{{\EuScript B}}

\def\a{\alpha}
\def\b{\beta}
\def\g{\gamma}

\def\D{\Delta}
\def\e{\varepsilon}

\def\l{\lambda}

\def\s{\sigma}

\def\z{\zeta}

\newcommand{\nc}{\newcommand} \newcommand{\renc}{\renewcommand}

\newcommand{\rdots}{\mathinner{ \mkern1mu\raise1pt\hbox{.}
    \mkern2mu\raise4pt\hbox{.}
    \mkern2mu\raise7pt\vbox{\kern7pt\hbox{.}}\mkern1mu}}

\def\bot{{\mathrm{bot}}}

\def\reg{{\mathrm{reg}}}

\def\bim{{\text{-bim}}}

\DeclareMathOperator{\Tr}{Tr}

\def\un{\underline}

\def\to{\rightarrow}
\def\from{\leftarrow}
\def\longto{\longrightarrow}

\def\onto{\twoheadrightarrow}
\nc{\triright}{\stackrel{[1]}{\to}}
\nc{\longtriright}{\stackrel{[1]}{\longto}}

\nc{\ot}{\otimes}

\nc{\HotRR}{{}_R\mathcal{K}_R}
\nc{\HotR}{\mathcal{K}_R}
\nc{\excise}[1]{}
\nc{\defect}{\text{df}}
\nc{\h}[1]{\underline{H}_{#1}}

\nc{\ptau}{\tau}

\nc{\Ga}{\mathbb{G}_a} 
\nc{\Gm}{\mathbb{G}_m} 

\nc{\Perv}{{\mathbf{P}}}

\nc{\Hi}{{\mathrm{H}}}
\nc{\IH}{{\mathrm{IH}}}

\nc{\ic}{\mathbf{IC}}

\nc{\gl}{{\mathfrak{gl}}}
\renc{\sl}{{\mathfrak{sl}}}
\renc{\sp}{{\mathfrak{sp}}}

\nc{\HBM}{H^{BM}}
\nc{\id}{{id}}

 \DeclareMathOperator{\Hom}{Hom}

\DeclareMathOperator{\End}{End} 

\newtheorem{thm}{Theorem}[section]
\newtheorem{lem}[thm]{Lemma}

\newtheorem{prop}[thm]{Proposition}
\newtheorem{cor}[thm]{Corollary}
\newtheorem{claim}[thm]{Claim}

\theoremstyle{definition}
\newtheorem{defi}[thm]{Definition}
\newtheorem{ex}[thm]{Example}

\theoremstyle{remark}
\newtheorem{remark}[thm]{Remark}
\newtheorem{question}[thm]{Question}
\newtheorem{warning}[thm]{Warning}

\DeclareMathOperator{\Ext}{Ext}

\DeclareMathOperator{\Spec}{Spec}

\newcommand{\into}{\hookrightarrow}

\def\gr{\textrm{gr}}

\DeclareMathOperator{\Gr}{Gr}
\nc{\simto}{\stackrel{\sim}{\to}}

\nc{\coker}{coker}

\nc{\amp}{{\textrm{ample}}}

\nc{\changed}[1]{{\color{blue}#1}}

\begin{document}

\title{Local Hodge theory of Soergel bimodules}
\date{May ~4th, 2016.}

\author{Geordie Williamson} 
\address{Max-Planck-Institut f\"ur Mathematik, Bonn, Germany}
\email{geordie@mpim-bonn.mpg.de}
\urladdr{}

\begin{abstract} We prove the local hard Lefschetz theorem and
  local Hodge-Riemann bilinear relations for Soergel
  bimodules. Using results of Soergel and K\"ubel one may deduce an
  algebraic proof of the Jantzen conjectures. We observe that the Jantzen
  filtration may depend on the 
  choice of non-dominant regular deformation direction.
\end{abstract}

\maketitle

\begin{center}
 \emph{\small{ Dedicated to Ben and Yeppie.}}
\end{center}


\section{Introduction} \label{sec:introduction}

In this paper we show the local hard Lefschetz theorem for Soergel
bimodules, as conjectured by Soergel and Fiebig. A key new ingredient
is the study of intersection forms and a proof of local Hodge-Riemann bilinear
relations. These properties are interesting in
themselves, as a further example of the remarkable Hodge theoretic structure
present in Soergel bimodules (even when no geometry is obviously
present).
It is also important because, by work of Soergel and K\"ubel,
it may be used to give an algebraic proof of the Jantzen conjectures on
the Jantzen filtration on Verma modules. (The first proof of the
Jantzen conjectures was given by Beilinson and Bernstein \cite{BB}.)

In geometric situations Soergel bimodules may be obtained as the
equivariant intersection cohomology of Schubert varieties. In this
setting the local hard Lefschetz theorem and local Hodge-Riemann
relations for Soergel bimodules follow from the hard
Lefschetz theorem and Hodge-Riemann relations for equivariant intersection cohomology, 
applied to a punctured standard affine
neighbourhood of a torus fixed point. Hence the results of this paper
may be seen as a translation and proof of these Hodge
theoretic statements into the algebra of Soergel bimodules.

This paper is a sequel to \cite{EW}
by Ben Elias and the author. The main ideas for the proofs are already
contained in \cite{EW}. This paper, like \cite{EW}, draws much
motivation from de Cataldo and 
Migliorini's Hodge theoretic proof of the decomposition theorem \cite{dCM,dCM2}.

\subsection{The fundamental example} We start by recalling the
geometric setting that led Soergel and Fiebig to the local hard Lefschetz conjecture. It is based
on \cite[Chapter 14]{BeL}, where Bernstein and
Lunts call this setting the ``fundamental example''. For us the name is
very appropriate: although all the proofs of this paper are
algebraic, all the motivation comes from the fundamental example.

Assume that $\CM^*$ acts linearly on $\CM^n$ with positive weights (i.e. 
$\lim_{z \to 0} z \cdot v = 0$ for all $v \in \CM^n$). Let $X \subset
\CM^n$ denote a closed $\CM^*$-stable subvariety. 
 Let $\Hi^*_{\CM^*}(pt;\RM)$ denote the $\CM^*$-equivariant cohomology of a point, which we
identify with $\RM[z]$, where $z$ ``is'' the first Chern class (of degree 2). When we come to discuss Soergel bimodules the choice of
coefficients in the real numbers will be important. When discussing
the fundamental example we could take coefficients in any field of
characteristic zero. To simplify notation we take coefficients in
the real numbers throughout.

Let $\IH^*$ (resp. $\IH^*_{\CM^*}$, resp. $\IH^*_{\CM^*,c}$) denote
(equivariant, compactly supported) intersection cohomology. A basic
fact is that we have a short exact sequence of graded $\RM[z]$-modules
\begin{equation} \label{eq:ihseq}
0 \to \IH_{\CM^*,c}^*(X) \to \IH_{\CM^*}^*(X) \to \IH^{*+1}(\dot{X}/\CM^*)
\to 0
\end{equation}
where $\dot{X} := X \setminus \{ 0 \}$ and $\IH^*(\dot{X}/\CM^*)$ is
a graded $\Hi^*_{\CM^*}(pt) = \RM[z]$ module via the identification
\[
\IH^{*+1}(\dot{X}/\CM^*) = \IH^*_{\CM^*}(\dot{X})
\]
which holds because $\CM^*$ acts on $\dot{X}$ with finite
stabilisers.

The above sequence is obtained by
taking equivariant hypercohomology of the standard (``Gysin'') distinguished triangle for
the equivariant intersection cohomology sheaf on $X$ with respect to
the decomposition $X = \{ 0 \} \sqcup \dot{X}$.
The first and
second terms of \eqref{eq:ihseq} can be identified with the
hypercohomology of the costalk
and stalk of the intersection cohomology sheaf at $0 \in X$  respectively. The resulting
long exact sequence yields the short exact sequence \eqref{eq:ihseq}
by purity, which ensures that all connecting
homomorphisms are zero.

To lighten notation we set $M^! := \IH_{\CM^*,c}^*(X)$, $M :=
\IH_{\CM^*}^*(X)$ and $H := \IH^*(\dot{X}/\CM^*)$ so that our
sequence takes the form
\begin{equation} \label{eq:ihseq2}
0 \to M^! \to M \to H[1] \to 0.
\end{equation}
($H[1]$ denotes a degree shift: $H[1]^i = H^{i+1}$).
Important ingredients in the fundamental example are the following 
facts about the sequence \eqref{eq:ihseq2}:
\begin{enumerate}
\item $M^!$ (resp. $M$) is a finitely generated free
  $\RM[z]$-module (as follows from purity) generated in degrees $> 0$
  (resp. $< 0$) (a consequence of the degree bounds on the stalks and costalks of intersection cohomology
  complexes.)
\item For all $i \ge 0$ multiplication by $z^i$ induces an isomorphism
\[
z^i : H^{-i} \simto H^{i}.
\]
(Indeed, the operator of multiplication by our generator $z
\in \Hi^2_{\CM^*}(pt)$  on $\IH^*_{\CM^*}(\dot{X}) =
  \IH^*(\dot{X}/\CM^*)[-1]$ may be identified, up to a non-zero scalar,
  with the action of the Chern class of the closed embedding $\dot{X} / \CM^*
  \hookrightarrow (\CM^n \setminus \{ 0 \} ) / \CM^* = \mathbb{\PM}$,
  into a
  weighted projective space. Now the result follows by the hard
  Lefschetz theorem for intersection cohomology.)
\item As the intersection cohomology of a projective variety, $H$ is
  equipped with a non-degenerate graded intersection pairing
  \[ \langle -, - \rangle : H \times H \to \RM.\] Moreover, for each $i
  \ge 0$, the ``Lefschetz'' form on $H^{-i}$ given by $(h, h') := \langle h, z^i
  h'\rangle$ (non-degenerate by (2)) induces a Hermitian form on $H^{-i} \otimes_{\RM} \CM$
  whose signature is governed by the Hodge-Riemann bilinear relations.
\end{enumerate}

This paper is concerned with establishing algebraic analogues of (2)
and (3) in the setting of Soergel bimodules. The bimodule analogue of
(1) is Soergel's conjecture, which was established in
\cite{EW}.

\subsection{Results}
Let $(W,S)$ denote a Coxeter system.
Let $\hg$ denote the reflection faithful
  representation of  $(W,S)$ over $\RM$ described in \S\ref{h}, and let $\{
  \a_s \} \subset \hg^*$ and $\{
  \a_s^\vee \} \subset \hg$ denote the simple roots and coroots (see \S\ref{h}). Let
$R$ denote the symmetric algebra on $\hg^*$ with $\deg \hg^* = 2$. Let
$\BC$ denote the category of Soergel bimodules (see \S\ref{sbim}). For
any $y \in W$, let $B(y)$ denote the indecomposable self-dual Soergel
bimodule parametrized by $y$.

Let $B$ denote a Soergel bimodule and fix $x \in W$. Define $B_x^!$
(resp. $B_x$) to be the largest submodule
(resp. largest quotient) of $B$ on which we have the relation $b \cdot
r = x(r) \cdot b$ for all $r \in R$. Then $B_x^!$ and $B_x$ are free
left $R$-modules. If $B$ is indecomposable and self-dual then $B_x^!$ (resp. $B_x$)
is generated in degrees $> 0$ (resp. $<0$) and their graded ranks are
given by Kazhdan-Lusztig
polynomials (Soergel's conjecture). Inclusion followed by projection
gives a canonical map
\[
i_x : B_x^! \into B \onto B_x.
\]
Moreover $i_x$ is an isomorphism over $Q$, the localisation of $R$ at all
roots. 

Any $\z^\vee \in \hg$ yields a specialisation $R \to \RM[z]$ given on
degree 2 elements by $\a \mapsto \langle \a, \z^\vee \rangle z$
(``restriction to the line $\RM\z^\vee \subset \hg$'').

\begin{thm} (``Local hard Lefschetz'') \label{lhl}
  Suppose that $\rho^\vee \in \hg$ is dominant (i.e. $\langle \a_s, \rho^\vee
  \rangle > 0$ for all $s \in S$) and that $B$ is indecomposable and
  self-dual. Define $H[1]$ as the cokernel of the inclusion:
\[
0 \to \RM[z] \otimes_R B_x^! \stackrel{i_x}{\longto} \RM[z] \otimes_R
B_x \to H[1] \to 0.
\]
Then $H$ satisfies the hard Lefschetz theorem: multiplication by
$z^{i}$ yields an isomorphism $H^{-i} \to H^{i}$ for all $i \ge 0$.
\end{thm}

This result was conjectured by Soergel \cite[Bemerkung 7.2]{Soe3} and
Fiebig \cite[Conjecture 6.2]{FCom}, motivated (as we 
will explain below) by the
fundamental example applied to the link of a singularity in a
Schubert variety. In fact they conjectured the theorem to hold for any
$\z^\vee \in \hg$ such that $\langle \a, \z^\vee \rangle \ne 0$ for
any root $\a$. The conjecture is false in this
generality. We will explain below that for Weyl groups its failure is related to the failure of semi-simplicity of the
layers of the Jantzen filtration associated to certain non-dominant
regular deformation directions.

The local hard Lefschetz theorem is the only geometric ingredient in
Fiebig's bound for the exceptional characteristics occurring in
Lusztig's conjecture (see \cite[\S 1.2]{Fiebound}). Using the above theorem one
can deduce the results of \cite{Fiebound} without recourse to geometry.

We now discuss the Hodge-Riemann bilinear relations. Suppose that $B$
is indecomposable and self-dual. Then $B$ carries an intersection
form\footnote{Throughout
  this paper, \emph{form} always means symmetric bilinear form.} 
\[
\langle -, -\rangle_B : B \times B \to R
\]
which is graded, symmetric and non-degenerate. (This is the analogue
of the equivariant intersection pairing in equivariant cohomology.)

Restricting $\langle -, -\rangle_B$ to $B_x^! \subset B$,
extending scalars to $Q$ and using that $i_x$ gives us a
canonical identification $Q \otimes_R B_x^! = Q \otimes_R B_x$ we obtain a
symmetric and $R$-bilinear $Q$-valued form
\[
\langle -, -\rangle^x_B : B_x \times B_x \to Q.
\]
This is the \emph{local intersection form} on $B_x$; it is the main
object in this paper.

Let $\rho^\vee \in \hg$ be dominant as above, and let $R \to \RM[z]$
denote the corresponding specialisation. To simplify notation, set
\[
N := \RM[z] \otimes_R B_x.
\]
Because $\langle \rho^\vee, \a \rangle \ne 0$ for any root $\a$, $\rho^\vee$
also induces a specialisation $Q \to \RM[z^{\pm 1}]$ and the
local intersection form $\langle -, -\rangle_B^x$ induces a symmetric 
$\RM[z]$-bilinear $\RM[z^{\pm 1}]$-valued form $\langle -, -\rangle_{N}$
on $N$.

The Hodge-Riemann bilinear relations give the signatures of the
restrictions of these forms to any homogeneous component of $N$. For $i
> 0$ set
\[
P^{-i} := (\deg_{< -i}N)^{\perp} \cap N^{-i}.
\]
(Here $\deg_{< i}N$ denotes the submodule of $N$ generated by all
elements of degree $< i$.)
Then the hard Lefschetz theorem implies that we have a decomposition
\[
N = \bigoplus_{i > 0} \RM[z] \otimes_{\RM} P^{-i}
\]
which is orthogonal with respect to $\langle -, -\rangle_N$. Let
$\min$ denote the minimal non-zero degree of $N$.

\begin{thm} (``Local Hodge-Riemann bilinear relations'') \label{lhr}
For any $i > 0$ the restriction of the $\RM$-valued form
  $(n,n') := z^i\langle n, n' \rangle_{N}$ on $N^{-i}$ to $P^{-i}$ is
  $(-1)^{\ell(x)}(-1)^d$-definite, where $d = (-i - \min)/2$.
\end{thm}

(The module $N$ vanishes unless $i$ and $\min$ are congruent modulo 2,
and hence the sign makes sense.)

Let us try to explain what the Hodge-Riemann bilinear relations mean
concretely for the local intersection forms $\langle -,
-\rangle_B^x$. Fix a graded basis $e_1, e_2, \dots, e_m$ for
$B_x$ as a left $R$-module, such that $\deg e_1 \le \deg e_2 \le \dots
\le \deg e_m$. We can think of the Gram\footnote{Throughout this paper
  \emph{Gram matrix} means the (symmetric) matrix of
  a (symmetric) form in some basis.} matrix $( \langle e_i, e_j \rangle_{B}^x )_{1 \le
  i, j \le m}$ as giving us a non-degenerate symmetric form on the
trivial vector bundle of rank $m$ over $\hg_\reg := \Spec
Q$. Moreover, this vector bundle is naturally filtered by the
subspaces generated by $\{ e_i \}_{\deg e_i \le d}$. In other words, we
can think of our form as a form on a filtered vector bundle. The
Hodge-Riemann bilinear relations predict the signatures of the
restriction of our form to all steps
of the filtration over the dominant regular locus 
\[\hg_\reg^+ := \{ \l^\vee \in \hg_\reg \; | \; \langle \a_s, \l^\vee
\rangle > 0 \text{ for all }s \in S \} \subset \hg_\reg.\]
Roughly speaking the signs must alternate at each step in the filtration.

For example, if the graded rank of $B_x$ is given by $v^{-5} + 3v^{-3}
+ 2v^{-1}$ and $\ell(x)$ is even, then the signs alternate as follows:
{\small}
\begin{equation*}
\left ( \begin{matrix} + & & & & & \\
& - & & & &\\
& & - & & &\\
& & & - & &\\
& & & & + & \\
& &  & & & +\end{matrix} \right )
\end{equation*}
If $\ell(x)$ is odd then the signs are given by $- + + + - -$.

Finally, there is one entry of the local intersection form which is
canonical. If $B$ is indecomposable and self-dual, then $B \cong B(y)$
for some $y \in W$, the smallest non-zero degree of $B(y)_x$ is
$-\ell(y)$ and $B(y)_x^{-\ell(y)}$ is generated by an element
$c_{x,y}$ which is well-defined up to a non-zero scalar. Our final result calculates the pairing of this element
with itself (see Theorem \ref{thm:em}):

\begin{thm} \label{em}
  $\langle c_{x,y}, c_{x,y} \rangle^x_{B(y)} = \g e_{x,y}$ for some
  $\g \in \RM_{>0}$.
\end{thm}

Here $e_{x,y}$ is the ``equivariant multiplicity'', a certain
homogenous rational function in $Q$ given by an explicit formula in
the nil Hecke ring.

\subsection{Relation to the fundamental example} \label{relfund}
Let us briefly comment on the connection between our results and the
fundamental example.

Let $G \supset B \supset T$ denote a complex reductive algebraic
group, a Borel subgroup and maximal torus, and let $(W,S)$ denote its
Weyl group and simple reflections. If we set $X_*$ and $X^*$ to be
the cocharacter and character lattice of $T$ then we can take $\hg := \RM
\otimes_{\ZM} X_*$ and $\hg^* := \RM \otimes_{\ZM} X^*$. The Borel
homomorphism gives us a canonical identification $R = S(\hg^*) = \Hi^*_T(pt)$.
 
Given any $y \in W$ we can consider the Schubert variety
$Z_y := \overline{ByB/B} \subset G/B$. By a theorem of Soergel 
\cite[\S 3.4]{SoergelLanglands}
 we may identify $B(y)$
with the equivariant intersection cohomology $\IH^*_T(Z_y)$. The
bimodule structure comes from the fact that $\IH^*_T(Z_y)$ is a module
over $\Hi^*_T(G/B) = R \otimes_{R^W} R$.

If $B := B(y)$ then the $R$-modules $B_x^!$ and $B_x$ can be described
as the $T$-equivariant cohomology of the costalk and stalk of the
intersection cohomology complex of $Z_y$ at the torus fixed
point $xB/B \in G/B$. Moreover, any choice of homomorphism $\g^\vee : \CM^* \to T$ yields a
line $\RM \g^\vee \subset \hg$, hence a specialisation $R \to \RM[z]$
and one may obtain the equivariant
cohomology (with respect to the induced $\CM^*$-action) of
the stalk and costalk via extension of scalars.

Now each $T$-fixed point $xB/B$ in $Z_y$ has a unique
$T$-stable affine neighbourhood $X_{x,y}$. We deduce from the
exact sequence \eqref{eq:ihseq} that if $\g^\vee : \CM^* \to T$ is such that the
induced action of $\CM^*$ on $X_{x,y}$ is attractive, then we have
\[
H = \IH^*(\dot{X}_{x,y}/\CM^*)
\]
where $\dot{X}_{x,y} := X_{x,y} - \{ xB/B \}$.
The hard Lefschetz and Hodge-Riemann relations now follow from the hard
Lefschetz and Hodge-Riemann bilinear relations in intersection
cohomology.

The need to reduce from the $T$-action to a
$\CM^*$-action to apply the fundamental example corresponds to the
choice of cocharacter $\rho^\vee \in \hg$ in Theorems \ref{lhl} and
\ref{lhr}. If one choses a cocharacter $\CM^* \to T$ such that the
induced action on $X_{x,y}$ is no longer attractive but is still
regular (i.e. $X_{x,y}^{\CM^*} = xB/B$) one still has
\[
H[1] = \IH^*_{\CM^*}( \dot{X}_{x,y})
\]
but now there is no longer any reason why $H$ should satisfy hard Lefschetz or
the Hodge-Riemann bilinear relations, because we cannot identify $H$
with the intersection cohomology of a projective variety. We will see
below that hard Lefschetz does indeed fail for certain specialisations
corresponding to regular (i.e. $\langle \alpha, \g^\vee \rangle \ne 0$
for all roots $\alpha$) non-dominant $\g^\vee \in \hg$.

\subsection{The Jantzen filtration} We conclude the introduction with
a discussion of
how our results are connected to the Jantzen filtration and
conjectures.

Let $\mathfrak{g}^\vee \supset \bg^\vee \supset \tg^\vee$ denote a complex
semi-simple Lie algebra, Borel subalgebra and Cartan subalgebra. (The
notation is intended to suggest that this data should be
Langlands dual to that of \S\ref{relfund}.) Given any weight $\l \in
(\tg^\vee)^*$ we can consider $\D(\l)$, the corresponding Verma module. It is
generated by a highest weight vector $v_\lambda$ which satisfies
\[ h \cdot v_\l = \l(h)v_\l  \qquad \text{for all $h \in \tg^\vee$}.\]

Given a deformation direction $\gamma \in (\tg^\vee)^*$ one can
consider the deformed Verma module $\Delta_{\CM[z]}(\lambda)$ which
is a $(\mathfrak{g}, \CM[z])$-bimodule generated by a highest weight vector
$v_\lambda$ satisfying
\[ h \cdot v_\l = (\l(h) + z\g(h))v_\l  \qquad \text{for all $h \in
  \tg^\vee$}.\]
That is, $\Delta_{\CM[z]}(\lambda)$ is a ``deformation of
$\Delta(\lambda)$ in the direction $\gamma$''. The deformed Verma module $\Delta_{\CM[z]}(\lambda)$ admits a unique
$\CM[z]$-bilinear contravariant form which specialises at $z = 0$ to
the contravariant form on $\Delta(\lambda)$. On
$\Delta_{\CM[z]}(\lambda)$ one has a filtration by order of vanishing
of the form, and if one considers the
specialisation at $z = 0$ one
obtains the \emph{Jantzen filtration}
\[
\dots \subset J^1  \subset  J^0 = \Delta(\lambda)
\]
which is exhaustive if $\g$ is regular.

The Jantzen conjectures \cite[5.17]{Jantzen} are the statements (for
deformation direction
$\g = \rho$, the half sum of the positive roots):
\begin{enumerate}
\item Certain canonical maps (e.g. embeddings $\Delta(\mu)
  \into \Delta(\lambda)$) are strict for Jantzen filtrations (see \cite[5.17, (1)]{Jantzen}).
\item The Jantzen filtration coincides with the socle filtration.
\end{enumerate}
(In \cite[5.17]{Jantzen} both statements are questions rather than
conjectures, and (1) is given more weight than (2).) It was
subsequently realised that the Jantzen conjectures have remarkable consequences: Gabber and
Joseph \cite{GJ} showed that (1) implies the Kazhdan-Lusztig conjectures on
multiplicities of simple modules in Verma modules (in a stronger form:
the Kazhdan-Lusztig polynomials give multiplicities
in the layers of the Jantzen filtration). Building on the work of
Gabber and Joseph, Barbasch \cite{Bar} showed that (1)
implies (2).

The Jantzen conjectures were proved by Beilinson and Bernstein in
\cite{BB}. They prove that the Jantzen filtration corresponds
under Beilinson-Bernstein localisation to the weight filtration on a standard
$D$-module. Part (1) of the Jantzen conjectures follows from the fact that any morphism between mixed perverse sheaves
strictly preserves the weight filtration. Part (2) follows via
a pointwise purity argument.

\subsection{The approach of Soergel and K\"ubel}
An alternative (``Koszul dual'') proof of the Jantzen conjectures was
initiated by Soergel \cite{SA} and completed by K\"ubel
\cite{K1,K2}. Recall that, by results of Soergel (see \cite{So}),
the principal block $\mathcal{O}_0$ of
category $\mathcal{O}$ is equivalent to (ungraded) modules over a
graded algebra $A_{\mathcal{O}}$. Soergel's conjecture is equivalent to
the fact that $A_{\mathcal{O}}$ may be chosen positively graded and semi-simple in degree
zero.

If one instead considers \emph{graded}
modules over $A_{\mathcal{O}}$ then one obtains a graded enhancement of
the principal block of $\mathcal{O}$.
It is known that Verma modules are gradable; that
is, the corresponding $A_{\mathcal{O}}$-modules admit gradings. Taken
together, the results of Soergel and K\"ubel show that the Jantzen
filtration on a Verma module agrees with the degree filtration on a
graded lift. Then part (1) of the Jantzen conjectures is immediate,
because the canonical maps in question can be lifted to maps of graded
modules. Part (2) follows because the socle, radical and degree filtrations for the graded
lifts of Verma modules coincide. (Once one knows that the degree zero part of $A_{\mathcal{O}}$ is
semi-simple, that $A_{\mathcal{O}}$ is generated in degrees
 $\le 1$, and that the head and socle of a Verma module is simple,
this follows from a simple observation about modules over graded
algebras \cite[Proposition 2.4.1]{BGS}.)

We now explain in more detail how the link between the
Jantzen and grading filtrations is made. Let $W$ denote the Weyl
group of $\mathfrak{g}^\vee \supset \tg^\vee$. For $x \in W$, denote by $\Delta(x)$
and $\nabla(x)$ the Verma and dual Verma modules of highest weight
$x(\rho) - \rho$. Let $T$ denote an indecomposable tilting module
in $\mathcal{O}_0$. 
Then $\Delta(x)$, $\nabla(x)$ and $T$ all admit ``deformations in
the direction $\g$'' over $\CM[[z]]$. (We need to pass from $\CM[z]$
to its completion $\CM[[z]]$ to apply idempotent lifting arguments.)
We denote these deformations by $\Delta_{\CM[[z]]}(x)$,
$\nabla_{\CM[[z]]}(x)$ and $T_{\CM[[z]]}$. Consider the
canonical pairing
\begin{align*}
\Hom(\Delta_{\CM[[z]]}(x),T_{\CM[[z]]}) \times &
\Hom(T_{\CM[[z]]},\nabla_{\CM[[z]]}(x)) \to
\Hom(\Delta_{\CM[[z]]}(x),\nabla_{\CM[[z]]}(x))  
\end{align*}
which lands in $\Hom(\Delta_{\CM[[z]]}(x),\nabla_{\CM[[z]]}(x))  = \CM[[z]]$.
As in the definition of the Jantzen filtration, we may define a
filtration on 
$\Hom(\Delta_{\CM[[z]]}(x),T_{\CM[[z]]})$ via order of vanishing. Upon
 specialisation at $z = 0$, we obtain the \emph{Andersen filtration}
\[
\dots \subset F^{i+1}  \subset  F^i \subset \dots \subset F^0 = \Hom(\Delta(x), T)
\]
which is exhaustive if $\g$ is regular.

Because everything in sight is free over $\CM[[z]]$ we can view the
pairing defining the Andersen filtration instead as an inclusion (for
$\g$ regular)
\[
\Hom(\Delta_{\CM[[z]]}(x),T_{\CM[[z]]}) \stackrel{\kappa}{\into} \Hom(T_{\CM[[z]]},\nabla_{\CM[[z]]}(x))^*
\]
where $*$ means $\CM[[z]]$ dual. Now the Andersen
filtration is obtained as the specialisation at $z = 0$ of the filtration:
\[
\dots \subset \kappa^{-1}(z^{i+1}
\Hom(T_{\CM[[z]]},\nabla_{\CM[[z]]}(x))^*) \subset \kappa^{-1}(z^i
\Hom(T_{\CM[[z]]},\nabla_{\CM[[z]]}(x))^*)  \subset \dots
\]
In \cite[\S 10.2]{SA}, Soergel identifies the inclusion $\kappa$ 
defining the Andersen filtration with the inclusion (now over $\CM$ rather than $\RM$)
\[
\CM[[z]] \otimes_R B_x^! \stackrel{\CM[[z]] \otimes i_x}{\xrightarrow{\hspace*{1.5cm}} } \CM[[z]] \otimes_R B_x
\]
appearing in the statement of local hard Lefschetz. Here $B$ is an
indecomposable self-dual Soergel bimodule such that $\widehat{B} =
\mathbb{V}T_{\hat{S}}$, where $\widehat{B}$ denotes the completion along the
grading of $B$, $\mathbb{V}$ is Soergel's structure functor
\cite[\S 5.10]{SA},
$T_{\hat{S}}$ denotes an $\hat{S}$-deformation of $T$ \cite[\S
3.5]{SA}, and $\hat{S}$
denotes the completion along the grading of $S(\mathfrak{t}^\vee)$ \cite[Theorem
8.2]{SA}.

Now comes the key observation: Theorem \ref{lhl} holds
  if and only if the filtration on $\CM[z] \otimes_R B_x^!$ given by
  \begin{equation}
    \label{eq:lfilt}
\dots \subset i_x^{-1}(z^{i+1} \otimes B_x) \subset i_x^{-1}(z^{i} \otimes B_x)
\subset  \dots    
  \end{equation}
induces the degree filtration on $\CM \otimes_{R}
B_x^!$. (A sketch: our inclusion $\CM[z] \otimes i_x$ is
  isomorphic to a direct sum of inclusions $N \into M$ of free graded
  $\CM[z]$-modules of rank 
  one. Now the above filtration
  induces the degree filtration on $\CM \otimes_R N$ if and only if $N$ and $M$ are
  generated in degrees symmetric about degree 0. See the last
  paragraph of \cite{SA}.)

In other words, local hard Lefschetz holds if and only if the
 Andersen and degree filtrations match under the identification
\[
\CM \otimes_R B_x^! = \Hom(\Delta(x), T).
\]

Finally, there is a contravariant ``tilting'' equivalence $t$ on the additive
category of modules with Verma flag, constructed by Soergel in
\cite{SoT}; it takes 
projective modules to tilting modules and sends $\Delta(x)$ to
$\Delta(-\rho - x(\rho))$. This equivalence
induces an isomorphism
\[
t : \Hom(P, \Delta(-\rho - x(\rho))) \stackrel{\sim}{\to} \Hom(\Delta(x), T)
\]
where $P :=
t^{^{-1}}(T)$ is a projective module. Now K\"ubel shows  (see \cite[Corollary 5.8]{K2} and
the discussion afterwards) that this isomorphism can be upgraded
to an isomorphism of graded vector spaces, matching the (filtration
induced by the) Jantzen filtration on $\Delta(-\rho - x(\rho))$ on
the left with the Andersen filtration on the right. By the discussion
above, the grading filtration on the right agrees with the Andersen
filtration, and hence the grading filtration on the left agrees with
the Jantzen filtration. This is enough to conclude that
  the grading and Jantzen filtrations agree on $\Delta(-\rho -
  x(\rho))$, which implies the Jantzen conjectures.

\subsection{Dependence on the deformation direction}
The statement of the local Hard Lefschetz theorem for Soergel
bimodules involved the choice of a specialisation parameter
$\rho^\vee \in \hg$. Similarly, the definition of the Jantzen
filtration involves the choice of a deformation direction $\gamma \in
(\tg^\vee)^*$. These choices match in the proof of Soergel and K\"ubel. In
particular, for an arbitrary (regular) specialisation $\rho^\vee \in \hg$, local hard
Lefschetz is equivalent to the Jantzen filtration agreeing with the
degree filtration. We have already commented that unless $\rho^\vee$ is
dominant, there is no geometric reason to expect hard Lefschetz to hold.

Using these observations we are able to answer the following
fundamental question about the Jantzen filtration (which seems to have
first been raised by Deodhar) in the negative:

\begin{question}(Deodhar, \cite[5.3, Bemerkung 2]{Jantzen})
  Is the Jantzen filtration independent of the choice of non-degenerate
  deformation direction $\g$?
\end{question}

Remarkably, this already fails for $\mathfrak{sl_4}(\CM)$. Using
Soergel bimodules and
Theorem \ref{em} one can see this failure via a simple
calculation in the nil Hecke ring. One can also verify this example
directly (without leaving the world of Verma modules). However here
the author needs computer assistance.

\subsection{Structure of the paper}

This paper is structured as follows:

\begin{description}
\item[\S\ref{nota}] We recall basic notation (shifts, gradings, degree
  filtration).
\item[\S\ref{sec:cox}] We recall the structure related to Coxeter
  groups underlying this paper (the reflection representation $\hg$, positivity
  properties, nil Hecke ring).
\item[\S\ref{fea}] We develop some algebra around the fundamental
  example (hard Lefschetz, Hodge-Riemann, weak Lefschetz substitute).
\item[\S\ref{mgs}] We develop the theory of \S\ref{fea} ``over
  $\PM^1$''. We define $\PM^1$-sheaves, study their structure and
  establish various conditions for their global sections to satisfy
  hard Lefschetz and Hodge-Riemann.  Although elementary, the results
  of this section are the main new ingredient in this paper.
\item[\S\ref{sbimb}] We give background on Soergel bimodules,
  define local intersection forms and establish some formulas for
  induced forms. We formulate a list of properties which form the
  ``local Hodge theory'' of Soergel bimodules.
\item[\S\ref{proof}] We prove the main results of the paper.
  \item[\S\ref{ap}] We outline an example in $\mathfrak{sl_4}(\CM)$
    where the Jantzen filtration does not coincide with the socle
    filtration.
\item[\S\ref{not}] Contains a list of the most important notation.
\end{description}

\subsection{Acknowledgements}
Although somewhat hidden, this paper makes heavy use of Peter
Fiebig's way of thinking about Soergel bimodules and Braden-MacPherson
sheaves. I would like to thank him for many hours of patient
explanations. Thanks to Roman Bezrukavnikov, Jens
Carsten Jantzen and Johannes K\"ubel for useful
discussions and to Wolfgang Soergel for many
  useful conversations and for providing the proof of
  Proposition \ref{Bxxs}. I
  would also like to thank the referee for a careful
  reading and many comments, which have contributed significantly to the readability and accuracy
  of this paper.
This paper started as joint work with Ben Elias, and can be seen as a
contribution to our joint long-term project of throwing some light on the Hodge
theoretic shadows cast by Soergel bimodules. To perform some of the
calculations in \S\ref{ap} the author had some help from magma \cite{magma}.

The results of this paper were announced in December 2013 at the Taipei Conference on
Representation Theory IV and in January 2014 in Patagonia. I would like to thank the
organisers of both conferences for the opportunity to present these results.

\section{Notation} \label{nota}




\subsection{Gradings and graded ranks} \label{grad}
Given a $\ZM$-graded object (vector space, module, bimodule) $M = \bigoplus
M^i$ we let $M[j]$ denote the shifted object with $M[j]^i = M^{i+j}$.  We call a
graded object $M$ \emph{even} if $M^{\textrm{odd}} = 0$ and \emph{odd} if
$M^{\textrm{even}} = 0$. We say that $M$ is \emph{parity} if it is
either even or odd. Given graded objects $M, M'$ we denote by
\[
\Hom^\bullet(M,M') = \bigoplus_{i \in \ZM} \Hom(M,M'[i])
\]
the (graded) space of homomorphisms of all degrees.

Let $R$ denote a polynomial ring which we view as a graded ring with
all generators of degree 2.  (Starting from \S\ref{sec:cox}, $R$ will
have a more specific meaning.) Given a graded free and finitely
generated $R$-module $M$ we can choose an isomorphism
\[
M \cong \bigoplus R[m]^{\oplus p_{m}}.
\]
We call $p = \sum p_{-m}v^m \in \ZM_{\ge 0}[v^{\pm 1}]$ the
\emph{graded rank} of $M$.

\subsection{Lattices and their duals} Let $Q$ denote a 
localisation of $R$ at some multiplicatively closed set of homogeneous
elements. Let $M_Q$ denote
a finitely generated graded free $Q$-module equipped with a non-degenerate
graded symmetric form 
\[
\langle -, - \rangle : M_Q \times M_Q \to Q.
\]
(Throughout non-degenerate means that $\langle -, - \rangle$ induces an
isomorphism $M_Q \simto M_Q^* = \Hom^\bullet_Q(M,Q)$, or alternatively that
the determinant of $\langle -, - \rangle$ in some basis is a unit
in $Q$.)
An $R$-submodule $M \subset M_Q$ is a \emph{lattice} if the natural map
$Q \otimes_R M \to M_Q$ is an isomorphism. If $M \subset M_Q$
is a lattice the \emph{dual lattice} is
\[
M^* := \{ m \in M_Q \; | \; \langle m, M \rangle \subset R \}.
\]
Then $M^*$ is canonically isomorphic to the dual of $M$ in the usual
sense (i.e. the natural map $M^* \simto
  \Hom^\bullet(M,R)$ is an isomorphism), $M^* \subset M_Q$ is a lattice and
$M = (M^*)^*$. In particular:
\begin{equation} \label{eq:pbar}
 \text{if $M$ has graded rank $p$ then $M^*$ has graded rank $\overline{p}$.}
\end{equation}
(By definition
$\overline{p}(v) = p(v^{-1})$ for a polynomial $p \in \ZM[v^{\pm 1}]$.)

\subsection{The degree filtration} \label{sec:deg}
Let $R$ be as in the previous
section. Let $M = \bigoplus M^j$ be a graded $R$-module. Set
\[
\deg_{\le i}M := R \cdot \bigoplus_{j \le i} M^{j}.
\]
This gives the \emph{degree filtration} 
\[
\dots \into \deg_{\le i}M \into \deg_{\le i+1}M
\into \dots\]
of $M$ by $R$-submodules. Any morphism of graded $R$-modules $f : M \to N$ 
preserves this filtration. Hence $\deg_{\le i}$ can be viewed as an endofunctor on
the category of graded $R$-modules.

\section{Coxeter group background} \label{sec:cox}

\subsection{Coxeter group} \label{sec:WS}

Let $(W,S)$ be a Coxeter system with length function $\ell : W \to
\ZM_{\ge 0}$ and Bruhat order $\le$.

An \emph{expression} $\un{x} = s_1 \dots s_m$ will denote a word in
$S$. Dropping the underline gives an element $x \in W$. An expression
$\un{x} = s_1 \dots s_m$ is \emph{reduced} if $\ell(x) = m$. Given an
expression $\un{x} = s_1 \dots s_m$, a \emph{subexpression} is a
sequence $\un{u} = t_1 \dots t_m$ such that $t_i \in \{ s_i, \id \}$ for
  all $i$. Again, dropping the underline denotes the product in
  $W$. For example, if $\un{y}$ is a reduced expression for $y$ then $\{ x \in W \;|\; x \le y \}  = \{ u \; | \; \text{$\un{u}$ is a
  subexpression of $\un{y}$}\}.$

\subsection{The reflection representation} \label{h}

We fix a realisation $(\hg, \hg^*, \{ \alpha_s \}, \{ \alpha_s^\vee \}
)$ of $(W,S)$ over $\RM$ as in \cite{Soe3, EW}. That is, $\hg$ is 
a finite dimensional real vector space and we have fixed linearly independent
subsets
\[
\{ \a_s \}_{s \in S} \subset \hg^* \text{ and } \{ \a^\vee_s \}_{s \in S} \subset \hg
\]
such that, for all $s, t \in S$, $\langle \alpha_s, \alpha_t^\vee \rangle = -2
  \cos(\pi/m_{st})$ (where $m_{st}$ denotes the order, possibly
  $\infty$, of $st  \in W$). In addition, we assume that $\hg$ is of minimal possible
  dimension satisfying the above two conditions. 
We can define an action of $W$ on $\hg$ via $s \cdot v = v - \langle
\alpha_s, v \rangle \alpha_s^\vee$ for all $s \in S$. This action is
(reflection) faithful \cite[Proposition 2.1]{Soe3}.


We consider the roots $\Phi := \bigcup_{w \in W} w \cdot \{ \alpha_s \} \subset \hg^*$
and coroots $\Phi^\vee := \bigcup_{w \in W} w \cdot \{ \alpha^\vee_s \} \subset
\hg$. We write $\Phi_+ \subset \Phi$ and $\Phi^\vee_+ \subset \Phi^\vee$ for the positive roots
and coroots. We have $\Phi = \Phi_+ \sqcup -\Phi_+$ and $\Phi^\vee = \Phi^\vee_+
\sqcup -\Phi^\vee_+$.

We write $T$ for the reflections (i.e.~conjugates of $S$) in $W$. We
have bijections
\[
T \simto \Phi_+ : t \mapsto \alpha_t \qquad T \simto \Phi^\vee_+ : t \mapsto \alpha_t^\vee
\]
such that $t(v) = v - \langle v, \alpha_t^\vee \rangle \alpha_t$ for
all $v \in \hg^*$.

Now let $\rho \in \hg^*$ be such that $\langle
\rho,\alpha_s^\vee \rangle > 0$ for all $s \in S$. (Such a $\rho$
exists because the set $\{ \alpha_s^\vee \}$ is linearly
independent.) Then we have
\begin{equation} \label{eq:xs>x}
tx > x \Leftrightarrow \langle x(\rho), \alpha_t^\vee \rangle > 0.
\end{equation}
Now fix $\rho^\vee$ in $\hg$ such that $\langle \alpha_s, \rho^\vee
\rangle > 0$ for all $s \in S$.

\begin{remark}
  The choice of $\rho$ and $\rho^\vee$ subject to the above positivity
  conditions is made arbitrarily and fixed throughout.
    This positivity property is used in a crucial way throughout this
    paper. We do not know if the results of \cite{EW} or this paper 
    are valid for an arbitrary reflection faithful representation of
    $(W,S)$.
\end{remark}

\begin{lem} \label{lem:rho}
  Suppose $x < w$. Then $\langle x(\rho), \rho^\vee \rangle > \langle w(\rho),
  \rho^\vee \rangle$.
\end{lem}

\begin{remark} One can view this lemma as saying that the map
\[
W \to \RM : w \mapsto - \langle w\rho,
\rho^\vee \rangle
\]
gives a refinement of the Bruhat order.\end{remark}

\begin{proof} By definition of the Bruhat order we may assume
  without loss of generality that
  $w = tx > x$ for some reflection $t \in T$. Then
\[
tx(\rho) = x(\rho) - \langle x(\rho), \alpha_t^\vee \rangle \alpha_t.
\]
Because $tx > x$ we know that $\langle x(\rho), \alpha_t^\vee \rangle
> 0$. Hence
\[
\langle w(\rho), \rho^\vee \rangle = \langle tx(\rho), \rho^\vee
\rangle = \langle x(\rho), \rho^\vee \rangle - \langle x(\rho),
\alpha_t^\vee \rangle \langle \alpha_t, \rho^\vee \rangle < \langle
x(\rho), \rho^\vee \rangle
\]
because $\langle \alpha_t, \rho^\vee \rangle > 0$. \end{proof}

\subsection{Positivity} \label{sec:pos}

From now on $R$ denotes the regular functions on $\hg$, or equivalently
the symmetric algebra $S(\hg^*)$ of $\hg^*$. We view $R$ as a graded
ring with $\deg \hg^* = 2$. Throughout $Q$ denotes the localisation of
$R$ at the multiplicatively closed subset generated by $\Phi$. In
formulas:
\begin{gather*}
  Q = R[1/\Phi].
\end{gather*}
By functoriality $W$ acts on $R$ and $Q$ via graded automorphisms.

For $s \in S$ we denote by $\partial_s$ the divided difference operator
\[
\partial_s(f) = \frac{f - sf}{\alpha_s}.
\]
Each $\partial_s$ preserves $R$. If $\l \in R$ is of
degree 2 then $\partial_s(\lambda)  = \langle \l, \alpha_s^\vee\rangle$.

We let $A = \RM[z]$ and $K = \RM[z^{\pm 1}]$, graded with $\deg z = 2$.

The map $\l \mapsto \langle \l, \rho^\vee \rangle z$ extends
multiplicatively to a morphism of graded rings
\[
\sigma : R \to A
\]
(``restriction to the line $\RM \rho^\vee \subset \hg$''). This map
(fixed by our choice of $\rho^\vee$) will play an important role
below. Whenever we write $A \otimes_{R} (-)$ we always mean that we view
$A$ as an $R$-module via $\sigma$.

Any homogenous element $f$ of $K$ is of the form $az^m$ for some $a \in
\RM$. We will write $f > 0$, $f < 0$ and say that $az^m$ is \emph{positive,
negative} etc. if $a$ is.

\subsection{The nil Hecke ring} \label{sec:nilhecke}

Let $Q_W$ denote the smash product of 
$Q$ with $W$. That is, $Q_W$ is a free left $Q$-module with basis $\{
\delta_w \; | \; w \in W \}$ and multiplication determined by
\[
(f \delta_x )(g \delta_y) = f(xg) \delta_{xy}.
\]
Inside $Q_{W}$ we consider the elements
\[
D_s := \frac{1}{\alpha_s} (\delta_\id - \delta_s) = (\delta_\id +
\delta_s) \frac{1}{\alpha_s}.
\]
The elements $D_s$ satisfy the following relations:
\begin{gather}
D_s^2 = 0; \label{eq:D^2} \\
\text{the $D_s$ satisfy the braid relations;} \label{eq:braid}\\
D_sf = (sf)D_s + \partial_s(f) \quad \text{for all $f \in Q$.}\label{eq:pol}
\end{gather}
If $\un{y} = st\dots u$ is a reduced
expression for $y \in W$ then, by \eqref{eq:braid},  we obtain well-defined elements
\[
D_y := D_sD_t \dots D_u  \in Q_W.
\]
We define rational functions $e_{x,y}$ for all $x, y$ through the identity:
\[
D_y = \sum e_{x,y}\delta_x.
\]
The rational functions $e_{x,y} \in Q$ are called \emph{equivariant
  multiplicities}. They are homogenous of degree $-2\ell(y)$.

\begin{remark}
The ring $Q_W$ acts naturally on $Q$ via $f\delta_x \cdot g
= fx(g)$. The \emph{nil Hecke ring} \cite{KK} is defined as the subring
$\{ q \in Q_W \; | \; q(R) \subset R \}$. In \cite[Theorem 4.6]{KK} it is shown
that the nil Hecke ring is a free left (or right) $R$-module with basis $\{ D_w
\; | \; w \in W \}$. Kostant and Kumar treat
  Weyl groups of Kac-Moody Lie algebras and take $Q$ to be the field
  of fractions of $R$, but their argument goes
  through in our setting.
\end{remark}

\begin{remark}
  For Kac-Moody groups, the $e_{x,y}$ describe the localisations at
  torus fixed points of the equivariant fundamental classes of
  Schubert varieties, hence their name. These functions were
  introduced by Kostant-Kumar and may be used to detect smoothness and
  rational smoothness \cite{Kumar, Brion} of Schubert varieties, as
  well as $p$-smoothness \cite{JW}.
\end{remark}

If $y's = y$ and $y' < y$ then expanding $D_y = D_{y'}D_s$ one obtains:
\begin{equation}
e_{x,y} = \frac{1}{x(\a_s)}(e_{x,y'} + e_{xs,y'}).\label{eq:xxs}
\end{equation}
The following well-known proposition provides a useful characterisation of equivariant
multiplicities:

\begin{prop} \label{prop:equi}
  The equivariant multiplicities are characterised by the following
  three properties:
  \begin{enumerate}
  \item We have $e_{x,y} = 0$ unless $x \le y$.
\item We have
\[
e_{y,y} = (-1)^{\ell(y)} \prod_{ t \in L_T(y)} \frac{1}{\a_t}
\]
where $L_T(y) = \{ t \in T \; | \; ty < y \}$.
\item Let $\un{y} = s_1 \dots s_m$ denote a reduced expression for
  $y$. Then for all $\l \in \hg^*$ and $x$ we have
\[
(x\l - y\l) e_{x,y} = \sum \langle s_{i+1}
\dots s_m \lambda, \alpha_{s_i}^\vee \rangle e_{x, y_{\hat{i}}}
\]
where the sum on the right hand side runs over those $1 \le i \le m$
such that $\un{y}_{\hat{i}} = s_1 \dots s_{i-1} s_{i+1} \dots
s_m$ is a reduced expression.
  \end{enumerate}
\end{prop}

\begin{proof}
  It is obvious that (1), (2) and (3) provide an inductive recipe for the
  computation of $e_{x,y}$. It remains to show that the claimed
  properties hold. Now (1) is immediate from the definition. For
  (2) consider a subexpression $\un{u}=t_1 \dots t_m$
    of $\un{y} =
  s_1 s_2 \dots s_m$ such that $u = t_1 \dots t_m = y$. Then
  clearly $t_i = s_i$ for all $i$ because $\un{y}$ is reduced. Hence
\begin{gather*}
e_{y,y}\delta_y = (-1)^{\ell(y)} \frac{1}{\alpha_{s_1}}\delta_{s_1} 
\frac{1}{\alpha_{s_2}}\delta_{s_2} \dots 
\frac{1}{\alpha_{s_m}}\delta_{s_m} =  \\
(-1)^{\ell(y)} \prod_{i = 1}^m \frac{1}{s_{1} \dots s_{i-1}(\alpha_i)}\delta_y 
= (-1)^{\ell(y)} \prod_{ t \in L_T(y)} \frac{1}{\a_t}\delta_y 
\end{gather*}
and (2) follows.

For (3) we can repeatedly apply \eqref{eq:pol} (using that
$\partial_s(\l) = \langle \l, \alpha_s^\vee \rangle$) to obtain the
equality (where $\;\widehat{}\;$ denotes omission)
\begin{equation} \label{eq:expandDs}
D_{s_1} \dots D_{s_m} \l  - (y\l) D_{s_1} \dots D_{s_m} = \sum_{i =
  1}^m \langle s_{i+1} \dots s_m \lambda, \alpha_{s_i}^\vee \rangle
D_{s_1} \dots \widehat{D_{s_i}} \dots D_{s_m}.
\end{equation}
If $s_1 \dots \hat{s_i} \dots s_m$ is not a reduced expression then $D_{s_1}
\dots \widehat{D_{s_i}} \dots D_{s_m} = 0$ by \eqref{eq:D^2} and
\eqref{eq:braid}. Now writing both sides of \eqref{eq:expandDs} in terms
of the basis $\delta_x$ gives the identity in (3).
\end{proof}

Recall the homomorphism $\s : R \to \RM[z]$ from \S \ref{sec:pos}.
The following positivity property of equivariant multiplicities will later
fix a sign ambiguity in the Hodge-Riemann bilinear relations:

\begin{cor} If $x \le y$ then $(-1)^{\ell(x)} \sigma(e_{x,y})> 0$.  \label{cor:HRamb}
\end{cor}

\begin{proof} We fix $x$ and induct on $\ell(y) - \ell(x)$. The 
  base case $x = y$ follows from Proposition \ref{prop:equi}(2)
  because all $\alpha_t$ appearing are positive, and hence $\s(\alpha_t) > 0$.
  
Now let $x < y$ and assume by induction that the proposition is
known for all $e_{x,y'}$ with $\ell(y') < \ell(y)$. Applying
Proposition \ref{prop:equi}(3) with $\l = \rho$ (our fixed element with $\langle
\rho, \a_s^\vee \rangle > 0$ for all $s \in S$) and multiplying by
$(-1)^{\ell(x)}$ we get the identity
\[
(-1)^{\ell(x)}\sigma(x\rho - y\rho)\sigma(e_{x,y}) = (-1)^{\ell(x)}\sum
\langle s_{i+1} \dots s_m \rho, \alpha_{s_i}^\vee \rangle \sigma (e_{x,y_{\hat{i}}}).
\]
Now by \eqref{eq:xs>x} the $\langle s_{i+1} \dots s_m \rho,
\alpha_{s_i}^\vee \rangle$ are all strictly positive, and by Lemma
\ref{lem:rho}, $\sigma(x\rho - y\rho)$ is strictly
positive. Induction now gives that the right hand side is
positive and the corollary follows.
\end{proof}

\section{Algebra around the fundamental example} \label{fea}

\subsection{Hard Lefschetz} \label{def:hl} Recall that $A = \RM[z]$
and $K = \RM[z^{\pm 1}]$, graded with $\deg z = 2$.

Let $N$ be a free finitely generated graded $A$-module generated in
degrees $\le 0$. We set $N_K := K \otimes_A N$ and assume that $N_K$
is equipped with a symmetric non-degenerate graded form
\[
\langle -, - \rangle: N_K \times N_K \to K.
\]

We say that $N$ \emph{satisfies hard Lefschetz} if the restriction
of $\langle -, - \rangle$ to $(\deg_{\le d} N)_K := K \otimes_A \deg_{\le d} N$ is non-degenerate for
all $d$.

It is immediate that $N$ satisfies hard Lefschetz if and only if
any of the following statements holds for all $d \le 0$:
\begin{enumerate}
\item the determinant of the Gram matrix of
the restriction of $\langle -, - \rangle$ to $\deg_{\le d} N$ is
non-zero ($\Leftrightarrow$ invertible in $K$);
\item the determinant of the Gram matrix of
the restriction of $\langle -, - \rangle$ to $N^d$ is
non-zero ($\Leftrightarrow$ invertible in $K$);
\item $\langle m, \deg_{\le d} N\rangle = 0$ for some $m
  \in \deg_{\le d} N$ implies $m = 0$;
\item the determinant of the Gram matrix of the form $\langle n, z^{-d}n' \rangle$ on
  $\deg_{\le d} N$ is non-zero ($\Leftrightarrow$ invertible in $K$).
\end{enumerate}

\begin{remark} \label{rem:P1hl}
  Condition (4) probably seems like a strange reformulation at this
  point. We have included it here, because it is this condition that
  will generalise to $\PM^1$-sheaves in the next section.
\end{remark}

Because $N$ is generated in degrees $\le 0$, the dual lattice
  $N^! \subset N_K$ is
  generated in degrees $\ge 0$ and hence $N^! \subset N$. Set
\[
 H := N/(zN^!).
\]
The following lemma (whose proof is an exercise) explains the terminology:

\begin{lem} \label{lem:zhl}
  $N$ satisfies hard Lefschetz if and only if for all $d \ge 0$
  multiplication  by $z^d : H^{-d} \to H^d$ is an isomorphism.
\end{lem}

\subsection{Hodge-Riemann} \label{def:hr}

Let $N$ be as in the previous section and assume that $N$ satisfies
hard Lefschetz. For $d \le 0$ define the \emph{primitive subspaces}: 
\[
P^d := N^d \cap (\deg_{\le d-1} N)^\perp \subset N^d.
\]

The following is an easy application of Gram-Schmidt orthogonalisation:

\begin{lem} \label{lem:prim1}
We have an orthogonal decomposition 
  $N = \bigoplus_{d \le 0} A \cdot P^d$.
\end{lem}

The following explains the ``primitive'' terminology:

\begin{lem}  \label{lem:prim}
We have a decomposition (as $A$-modules):
\[
H = \bigoplus_{d \le 0} A/(z^{-d+1}) \otimes P^d.
\]
\end{lem}

The restriction of our form to $N^d$ for $d \le 0$ takes values in
$K^{2d} = \RM z^d$ for degree reasons. Hence the \emph{Lefschetz
  form} $(n,n') \mapsto z^{-d}\langle
n, n' \rangle$ on $N^d$ takes values in $\RM$.

We say that $N$ \emph{satisfies HR} (short for ``satisfies the
Hodge-Riemann bilinear relations'') if:
\begin{enumerate}
\item $N$ is parity (i.e. $N$
vanishes in either odd or even degree);
\item if $\min$ denotes the minimal non-zero degree of $N$ then there exists $\e
  \in \{ \pm 1 \}$ such that, for all $d = \min + 2i \le 0$, the
  Lefschetz form
  $z^{-d}\langle p, p' \rangle$ on $P^{d}$ is $\e(-1)^i$-definite.
\end{enumerate}

  \begin{lem}
    \label{lem:HRsummand}
Suppose that $N$ satisfies HR and that $N = N' \oplus N''$ is an
orthogonal decomposition. Then $N'$ and $N''$ satisfy HR.
  \end{lem}

  \begin{proof}
    Fix $d \le 0$ and $p \in P^d$. Then we can write $p = p' + p''$
    with $p' \in N'$, $p'' \in N''$. Because $p$ is primitive and $N'$
    and $N''$ are orthogonal
\[
0 = \langle p, \deg_{< d} N' \rangle = \langle p', \deg_{<d }N'
\rangle = \langle p', \deg_{<d} N \rangle
\]
and hence $p' \in P^d$. Similarly, $p'' \in P^d$. Hence our
decomposition $N = N' \oplus N''$ induces a refinement of the
decomposition in Lemma \ref{lem:prim1} and the result follows, because
the restriction of a definite form to a subspace is definite of the
same sign.
  \end{proof}

 Suppose that $N$ is parity and generated in degrees $\le 0$, and that
 $\min$ denotes its minimal non-zero degree. Then we
  can write its graded dimension as $v^{\min}f(v^2)$ for some $f \in
  \ZM_{\ge 0}[v]$. 

\begin{lem} \label{lem:signs}
  $N$ satisfies HR if and
  only if there exists $\e \in \{ \pm 1 \}$ such that for all $i \ge
  0$ with $d =
  \min + 2i \le 0$ the form $\langle z^{-d}x, y \rangle$ on 
  $N^{d}$ has signature $\e(\tau_{\le i}f)(-1)$ (by definition $\tau_{\le d}(\sum \a_jv^j) = \sum_{j \le
    d}\a_jv^j$).
\end{lem}

\begin{proof}
  Let $d = \min + 2i \le 0$. The decomposition in Lemma
  \ref{lem:prim1} gives a decomposition
\[
N^d = z^iP^{\min} \oplus \dots \oplus zP^{d-2} \oplus P^d.
\]
This decomposition is orthogonal
for Lefschetz forms and $z : N^{d-2} \to N^d$ is an
isometry. The lemma now follows: fixing the signature of the
Lefschetz forms on $N^d$ for all $d \le 0$ is equivalent to fixing
the signature on $P^d$ for all $d \le 0$.
\end{proof}

\subsection{Weak Lefschetz} Let $N_K$ and $N'_K$ be two finitely
generated free graded $K$-modules equipped with non-degenerate
symmetric forms $\langle -, - \rangle$ and $\langle -, -
\rangle'$. Let $N \subset N_K$ and $N' \subset N'_K$ be lattices
generated in degrees $\le 0$.

The following proposition provides a useful tool for establishing hard
Lefschetz inductively (it is essentially a restatement of \cite[Lemma
2.3]{EW}):

\begin{prop} \label{prop:wL}
  (``weak Lefschetz substitute'')
Suppose that we have maps $d : N \to N'[1]$, $d' : N' \to N[1]$ such that:
\begin{enumerate}
\item $d, d'$ are adjoint (i.e. $\langle d(n), n' \rangle' =
    \langle n, d'(n') \rangle$ for all $n \in N$, $n' \in N'$);
\item $d' \circ d$ is equal to multiplication by $0 \ne \b \in A$.
\end{enumerate}
Then if $N'$ satisfies HR then $N$ satisfies hard Lefschetz.
\end{prop}

\begin{proof}
First note that $d$ is injective by (2). Now assume for contradiction that $\langle -, - \rangle$ does not
  satisfy hard Lefschetz. In other words, there exists $0 \ne m \in N$
  of
  degree $i \le 0$ such that $\langle m,  \deg_{\le i} N \rangle =
  0$. By assumption $\deg_{\le 0}N = N$ and $\langle -,
    -\rangle$ is non-degenerate, so we can assume $i < 0$.
Then $0 \ne d(m) \in (N')^{i+1}$ and for all $m' \in \deg_{\le i-1}
  N'$ we have
\[
\langle d(m), m' \rangle' = \langle m, d'(m') \rangle = 0
\]
because $d'(m') \in \deg_{\le i} N$. In particular, $d(m)$ is
orthogonal to $\deg_{\le i-1} N'$ (and even to $\deg_{\le i} N'$
because $(N')^i = 0$, as $N'$ satisfies HR and hence is parity). In
particular $d(m) \in P^{i+1} \subset
(N')^{i+1}$. Hence, as $N'$ satisfies HR, we have
\[
0 \ne \langle d(m), d(m) \rangle' = \langle m, (d' \circ d)(m) \rangle =
\beta \langle m, m \rangle
\]
which contradicts $\langle m,  \deg_{\le i} N \rangle =
  0$.
\end{proof}

\section{Moment graph sheaves on the projective line} \label{mgs}

In this section we study certain sheaves on the moment graph of
$\PM^1$, which we dub $\PM^1$-sheaves. This provides a useful language
for discussing certain local calculations with Soergel
bimodules.

\begin{remark}
  Although we do not discuss the general theory below, our
  discussion 
  has been strongly influenced by the Braden-MacPherson and Fiebig
  theory of sheaves on moment graphs \cite{BM,FCom, Fiebig}.
\end{remark}

\subsection{$\PM^1$-sheaves} \label{sec:P1} Let $A = \RM[z]$ and $K =
\RM[z^{\pm 1}]$ as
above.

\begin{defi} A \emph{sheaf on the moment graph of $\PM^1$} is a
    collection $M$ of
\begin{enumerate}
\item  finitely generated graded $A$-modules $M_0, M_\infty$ and $M_{\CM^*}$;
\item graded $A$-module morphisms $\rho_0 : M_0 \to M_{\CM^*}$, $\rho_\infty : M_\infty \to
  M_{\CM^*}$
\end{enumerate}
such that $M_{\CM^*}$ is annihilated by $z \in A$.
\end{defi}
The category of sheaves on the moment graph of $\PM^1$ is a graded (with shift functor $[1]$), additive
category in an obvious way.

\begin{defi} Let $M$ be a sheaf on the moment graph of $\PM^1$. We say
  that $M$ is a \emph{$\PM^1$-sheaf} if $M_0$ and $M_\infty$ are free
  $A$-modules, $\rho_0$ is surjective and
  $\rho_\infty$ is isomorphic to the quotient map $M_\infty \to M_{\infty}/(z)$.
\end{defi}


\begin{remark} Let $\CM^*$ act non-trivially and linearly on $\PM^1$.
Any object in the constructible $\CM^*$-equivariant derived category
of  $\PM^1$ yields modules 
$M_0$, $M_\infty$ and $M_{\CM^*}$ over $\Hi^*_{\CM^*}(pt) =
A$ by taking equivariant hypercohomology of the stalks at $0$,
$\infty$ and $\CM^*$ \cite{BM, FW}. This explains the name.
\end{remark}

\begin{remark} In Fiebig's language, $\PM^1$-sheaves are the
  Braden-MacPherson sheaves on the moment graph of $\PM^1$. However we
  prefer the term $\PM^1$-sheaf in this context because $\PM^1$-sheaves are quite
  simple objects (in contrast
  to Braden-MacPherson sheaves on general moment graphs).
\end{remark}

  The two most important examples of sheaves on the moment graph of
  $\PM^1$ are the skyscraper at 0 ($M_0 = A$, $M_{\CM^*} = M_\infty =
  0$) which we will call simply the \emph{skyscraper}, and the \emph{constant sheaf} ($M_0 = M_\infty = A$, $M_{\CM^*} =
  A/(z)$, $\rho_0, \rho_\infty$ the canonical quotient
  maps). Both are $\PM^1$-sheaves. In fact:

  \begin{lem} \label{lem:P1clas}
    Any $\PM^1$-sheaf is (non-canonically) isomorphic to a direct sum of shifts of skyscraper
    and constant sheaves. Hence any indecomposable  $\PM^1$-sheaf is isomorphic (up to shift) to a skyscraper or constant sheaf.
  \end{lem}

  \begin{proof} Exercise.
  \end{proof}

Let $M$ denote a $\PM^1$-sheaf.
The \emph{global sections} of $M$ are
\[
M_{0,\infty} := \{ (m_0, m_\infty) \in M_0 \oplus M_\infty \; | \; \rho_0(m_0) =
\rho_\infty(m_\infty) \} \subset (M_0 \oplus M_\infty)
\]
which we regard a left $A$-module via $r \cdot (m_0, m_\infty) =
(rm_0, rm_\infty)$. We have
\begin{equation}
  \label{eq:Ksplit}
  K \otimes_A M_{0,\infty} = K \otimes_A M_0
\oplus K \otimes_A M_\infty.
\end{equation}

  \begin{remark} \label{rem:Mrank}
By Lemma \ref{lem:P1clas} if the graded ranks of $M_0$ and $M_\infty$
are $p_0, p_\infty \in \ZM_{\ge 0}[v^{\pm 1}]$ respectively, then the graded rank of $M_{0,\infty}$ is
\[
(p_0 - p_\infty) + (1 + v^2)p_\infty = p_0 + v^2 p_\infty.
\]
\end{remark}

More generally we consider the \emph{structure algebra}
\[
Z := \{ (r_0, r_\infty) \in A \oplus A \; | \; r_0 = r_\infty
\textrm{ mod }
(z) \}.
\]
Of course this is nothing other than the global sections of the
constant sheaf. It is a ring via pointwise multiplication. Moreover, 
one may check that $Z$ acts on the global sections of any $\PM^1$-sheaf via $(r_0, r_\infty) \cdot (m_0, m_\infty) = (r_0 m_0, r_\infty
m_\infty)$ for $(m_0, m_\infty) \in M_{0,\infty}$.

Below a special role will be played by the action of degree 2 elements
of $Z$ on the global sections of $\PM^1$-sheaves (``Lefschetz operators''). Of
course
\[
Z^2 = \RM z \oplus \RM z.
\]
We define the \emph{ample cone} in $Z^2$ to be
\[
Z^2_\amp := \{(\l_0, \l_\infty) =  (az, bz) \in Z^2 \; | \; 0 < b < a \}.
\]


\subsection{Polarised $\PM^1$-sheaves} \label{sec:P1} Let $M$ be a $\PM^1$-sheaf.
  A \emph{polarisation} of $M$ is a pair of symmetric graded $K$-valued
  $A$-bilinear forms:
  \begin{align*}
    \langle -, - \rangle^0 &: M_0 \times M_0 \to K, \\
    \langle -, - \rangle^\infty &: M_\infty \times M_\infty \to K.
  \end{align*}
A polarisation is \emph{non-degenerate} if both $\langle -, -
\rangle^0$ and $\langle -, - \rangle^\infty$ are non-degenerate over
$K$. A \emph{polarised $\PM^1$-sheaf} is a
  $\PM^1$-sheaf together with a non-degenerate polarisation.

A polarisation of $M$ induces an $A$-bilinear form
\[
\langle -, - \rangle =   \langle -, - \rangle^0 + \langle -, -
\rangle^\infty : M_{0,\infty} \times M_{0,\infty} \to K
\]
on the global sections of $M$. We have
\begin{equation*}
  \label{eq:form-assoc}
  \langle \g m, m' \rangle = \langle m, \g m' \rangle
\end{equation*}
for all $m, m' \in M_{0,\infty}$ and $\g \in Z$.
By \eqref{eq:Ksplit} we see that over
$K$ the form $\langle -, - \rangle$ 
is just the direct sum of $\langle -, - \rangle^0$ and $\langle -, -
\rangle^\infty$. In particular, $\langle -, - \rangle$ is
non-degenerate if the polarisation is.


\subsection{Hard Lefschetz}
Let $M$ be a polarised $\PM^1$-sheaf.
We assume that the global sections of $M$ are generated in degrees $\le 0$.

We say that $\g \in Z^2$ \emph{satisfies hard Lefschetz} on $M$ if
and only if for all $d \le 0$ the form $\langle \g^{-d} x, y \rangle$ on
$\deg_{\le d} M_{0,\infty}$ is non-degenerate (i.e. the determinant of its Gram
matrix
 is invertible $\Leftrightarrow$ non-zero in $K$). We say that $M$ \emph{satisfies
  hard Lefschetz} if $\g$ satisfies hard Lefschetz on $M$ for all
$\g \in Z^2_{\amp}$.

\begin{remark}
  See Remark \ref{rem:P1hl} for some motivation for this definition.
\end{remark}

Recall that $M_{0,\infty}$ is equipped with a non-degenerate form
given by the sum of the forms $\langle -, -\rangle^0$ and $\langle -,
- \rangle^\infty$. Let $M_{0,\infty}^! \subset K \otimes_A M_0 \oplus
K \otimes_A M_\infty$ denote the dual
lattice. Because $M_{0,\infty}$ is generated in degrees $\le 0$,
$M_{0,\infty}^!$ is generated in degrees $\ge 0$ and so
$M_{0,\infty}^! \subset M_{0,\infty}$. We set
\[
H_{0,\infty} := M_{0,\infty} / (zM_{0,\infty}^!).
\]
Any $\g \in Z^2$ preserves $M_{0,\infty}$ and $M_{0,\infty}^!$ and
hence induces a degree 2 operator on
$H_{0,\infty}$. The above definition is equivalent to $\g$ satisfying
hard Lefschetz in the usual sense (i.e. $\g^i : H_{0,\infty}^{-i} \to
H_{0,\infty}^i$ is an isomorphism for all $i \ge 0$).

\begin{remark} \label{rem:P1global}
  The condition for the $\PM^1$-sheaf $M$ to satisfy hard Lefschetz is not the same as
  requiring that its global sections $M_{0,\infty}$ satisfy hard
  Lefschetz (in the sense of \S\ref{def:hl}). Indeed, $M_{0,\infty}$ satisfies hard Lefschetz if
  and only if $\g = (z,z)$ satisfies hard Lefschetz on
  $H_{0,\infty}$, whereas $M$ satisfies hard Lefschetz if and only if
  $(az,bz)$ satisfies hard Lefschetz on $H_{0,\infty}$, for all $0 < b
  < a$. Hence the 
  condition for the global sections $M_{0,\infty}$ to satisfy hard
  Lefschetz is a ``degeneration to a wall'' of
  $M$ satisfying hard Lefschetz.
\end{remark}

\begin{ex} \label{ex:constant}
  We consider the simplest non-trivial example. Let $M$ be a constant
  $\PM^1$-sheaf generated in degree $m$ for some $m \le -2$:
  $M_0 = M_\infty = A[-m], M_{\CM^*} = A/(z)[-m]$, $\rho_0 =
  \rho_\infty$ the quotient maps. (The condition $m \le -2$ is to ensure
    that the global sections are generated in degrees $\le 0$.)
Equip $M$ with the polarisation
\[
\langle 1, 1 \rangle^0 = \l_0 z^{m} \quad \text{and} \quad \langle 1,
1 \rangle^\infty = \l_\infty z^{m} \quad \text{for some }\l_0, \l_\infty \in \RM.
\]
We assume the polarisation is non-degenerate (i.e. $\l_0 \ne 0 \ne
\l_\infty$). The global sections of $M$ are
\[
M_{0,\infty} = A \cdot (1,1) \oplus A \cdot (z,0)
\]
with the generators in degrees $m$ and $m+2$ respectively. Hence:
\[
\deg_{\le d}M_{0,\infty} = \begin{cases} 0 & \text{if }d< m, \\ A \cdot
  (1,1) & \text{if }d = m, m+1, \\ M_{0,\infty} &\text{if }d \ge m+2. \end{cases}
\]
Let $\g =
(az, bz) \in Z^2$. We calculate the forms $\langle \g^{-d} x, y \rangle$ on $\deg_{\le d}
M_{0,\infty}$ in the above basis:
\begin{align*}
  d = m, m+1: & \quad ((\l_0a^{-d} + \l_\infty b^{-d})z^{-d+m}) \\
  m+2 \le d \le 0: &
\quad \left ( \begin{matrix} (\l_0a^{-d} + \l_{\infty} b^{-d})z^{-d+m} &
    \l_0a^{-d}z^{-d+1+m} \\ \l_0a^{-d}z^{-d+1+m} &
    \l_0a^{-d}z^{-d+2+m}  \end{matrix} \right)
\end{align*}
Calculating determinants we conclude that $\g$ satisfies hard Lefschetz on
$M$ if
and only if
\begin{gather*}
\l_0a^{-d} + \l_{\infty} b^{-d} \ne 0 \quad \text{for }d = m, m+1, \\
\l_0\l_{\infty} a^{-d}b^{-d} \ne 0 \quad \text{for } m+2\le d \le 0.
\end{gather*}
For $\g \in Z^2_{\amp}$ the second condition is automatic. The
first condition holds for all $\g \in Z^2_{\amp}$ (i.e. for all $0 < b
< a$) if and only if
either:
\begin{enumerate}
\item $\l_0, \l_\infty$ have the same sign, or
\item$\l_0$ and $\l_\infty$
have opposite signs and $|\l_0|\ge |\l_\infty|$.
\end{enumerate}
Below it will be the second case that is relevant. In case (2) the global sections satisfy hard Lefschetz if and
only if we have strict inequality $|\lambda_0| > |\lambda_\infty|$.
 This is an illustration of Remark \ref{rem:P1global}.
\end{ex}

\subsection{Hodge-Riemann} \label{sec:hr} Let $M$ be a polarised
$\PM^1$-sheaf as in the previous section (i.e. the global sections of $M$ are generated in degrees
$\le 0$).

Let $\g \in Z^2$ and assume that $\g$ satisfies hard Lefschetz on
$M$. For $d \le 0$ define the \emph{$\g$-primitive subspaces:}
\[
P^d_\g := (\g^{-d+1} \deg_{\le d - 1}M_{0,\infty})^\perp \cap M^d_{0,\infty} = 
(\g^{-d+1} M^{d-2}_{0,\infty})^\perp \cap M^d_{0,\infty} \subset M^d_{0,\infty}.
\]
The following is an easy application of Gram-Schmidt orthogonalisation:

\begin{lem}
  We have a decomposition $M_{0,\infty} = \bigoplus_{d \le 0}
  \RM[\g] \cdot P^d_\g.$
\end{lem}

\begin{warning}
  Unless $\g = (az, az)$ the subspaces $\RM[\g]
  \cdot P^d_\g \subset M_{0,\infty}$ are not $A$-submodules in general
  and the above decomposition need not be orthogonal (although it is
  orthogonal between degrees $d$ and $-d$).
\end{warning}

We say that $\g \in Z^2$ \emph{satisfies HR} on $M$ if:
\begin{enumerate}
\item $M_0$ and $M_\infty$ are either both even or both odd
  (hence the   global sections $M_{0,\infty}$ are either
  even or odd); 
\item 
if $\min$ denotes the minimal non-zero degree of $M_{0,\infty}$ then there exists $\e
  \in \{ \pm 1 \}$ such that, for all $d = \min + 2i \le 0$, the form
  $\langle \g^{-d}p, p' \rangle$ on $P_\g^{d}$ is $\e(-1)^i$-definite.
\end{enumerate}
We say that $M$ \emph{satisfies HR} if all $\g \in Z^2_\amp$
satisfy HR on $M$.

\begin{remark} \label{rem:HRsignature}
Because $M_{0,\infty}$ is parity we can write the
  graded rank of $M_{0,\infty}$ as $v^{\min} f(v^2)$ for some $f \in
  \ZM_{\ge 0}[v]$, where $\min \in \ZM$ denotes the minimal non-zero 
 degree of $M_{0, \infty}$. Fix $d \le 0$ with $d = \min + 2i \le
 0$. We have a decomposition
\[
M^d_{0,\infty} = P^d_\g \oplus \g P^{d-2}_\g \oplus \dots \oplus \g^iP^{\min}_\g.
\]
From the definitions it follows that this decomposition is orthogonal
with respect to the form $\langle x, \g^{-d} y\rangle$. Moreover, the
induced form on the subspace
\[
 \g P^{d-2}_\g \oplus \dots \oplus \g^iP^{\min}_\g \subset M^d_{0,\infty}
\]
agrees with the form $\langle x, \g^{-d+2}y \rangle$ on $M^{d-2}_{0,\infty}$
(i.e. $\g : M^{d-2}_{0,\infty} \to M^d_{0,\infty}$ is an isometry). In particular, we see
that $\g$ satisfies HR on $M$ if and only if there exists $\e \in \{
\pm 1 \}$ such that the signature of $\langle x, \g^{-d} y \rangle$ on
$M^{d}_{0,\infty}$ is $\e(\tau_{\le i}f)(-1)$ for all $\min \le d = \min + 2i \le
0$. (See Lemma \ref{lem:signs}.)
\end{remark}

\begin{remark}
  The form $\langle -, - \rangle$ on $M_{0,\infty}$ induces in a
  natural way an $\RM$-valued form on $H_{0,\infty} =
  M_{0,\infty}/(zM_{0,\infty}^!)$. Then $\g$ satisfies HR if and only
  if $\g$ induces a Lefschetz operator satisfying HR on $H_{0,\infty}$
  (in the usual sense).
\end{remark}

\begin{ex} \label{ex:constHR}
  We continue the example of the polarised constant sheaf begun in
  Example \ref{ex:constant}. The form $\langle x, \g^{-m} y \rangle$ on
  $M_{0,\infty}^{m}$ is $(\l_0a^{-m} + \l_\infty b^{-m})$. The form
  $\langle x, \g^{-m-2} y \rangle$ on $M_{0,\infty}^{m+2}$ in the
  basis $\{(z,0), (0,z) \}$ is
\[
\left ( \begin{matrix} \l_0a^{-m-2} & 0 \\ 0 & \l_\infty
    b^{-m-2} \end{matrix} \right ).
\]
For HR to be satisfied in degree $m+2$ this matrix must have signature
0. Hence if $\g \in Z^2_\amp$ then $\l_0$ and $\l_\infty$ must have
opposite signs. We conclude that $M$ satisfies HR if and only if $\l_0$ and
$\l_\infty$ have opposite signs, and $|\l_0| \ge
|\l_\infty|$. The global sections
$M_{0,\infty}$ satisfy HR if and only if $\l_0$ and $\l_\infty$ have
opposite signs and $|\l_0| > |\l_\infty|$.
\end{ex}

It is clear that if $\g$ satisfies hard Lefschetz or HR on $M$ then so
does any positive scalar multiple of $\g$. Hence the following lemma is
easy:

\begin{lem} \label{lem:saturation}
  Let $M$ denote a polarised $\PM^1$-sheaf whose global
    sections are generated in degrees $\le 0$.
Suppose that for all $1 < c$ there exists $0 < b < a$ such
  that $c = a/b$ and $\g = (az, bz)$ satisfies hard Lefschetz (resp. HR) on $M$. Then
  $M$ satisfies hard Lefschetz (resp. HR).
\end{lem}

\subsection{Weak Lefschetz}

The following is the analogue for $\PM^1$-sheaves of Proposition \ref{prop:wL}:

\begin{prop} \label{prop:wLP1} (``weak Lefschetz substitute for
  $\PM^1$-sheaves'') Let $M, M'$ be two polarised $\PM^1$-sheaves
  and fix $\g = (\l_0, \l_\infty) \in Z^2$ such that $\l_0, \l_\infty$
  are both non-zero. Assume that we are given morphisms $d : M
  \to M'[1]$ and $d' : M' \to M[1]$ such that:
  \begin{enumerate}
  \item $d$ and $d'$ are adjoint (i.e. $\langle dm, m' \rangle^? =
    \langle m, d'm' \rangle^?$ for $? \in \{ 0 , \infty \}$);
  \item $d' \circ d$ is equal to multiplication by $\g$.
  \end{enumerate}
Suppose that $\g$ satisfies HR on $M'$. Then $\g$ satisfies hard
Lefschetz on $M$.
\end{prop}

\begin{proof} As in the proof of Proposition
  \ref{prop:wL}, (2) implies that $d_0 : M_0 \to M'_0[1]$ and $d_\infty :
  M_\infty \to M'_{\infty}[1]$ are injective.

  Assume for contradiction that $\g$ does not satisfy hard Lefschetz
  on $M$. Then there exists $0 \ne m \in M_{0,\infty}$ of degree $-i$ for $i
  \ge 0$ such that
\begin{equation} \label{eq:orthog}
\langle \g^i m, \deg_{\le -i} M_{0,\infty} \rangle = 0.
\end{equation}
Because $\langle -, -\rangle$ on $M$ is non-degenerate and $\deg_{\le
  0} M_{0,\infty} = M_{0,\infty}$ we must have $i > 0$.
Then $0 \ne dm \in (M')^{-i+1}_{0,\infty}$ and for all $m' \in \deg_{\le -i -
  1}(M'_{0,\infty})$ we have
\[
\langle d(m), \g^im' \rangle = \langle m, \g^i d'(m') \rangle = \langle
\g^i m , d'(m') \rangle = 0.
\]
Hence $d(m)$ is orthogonal to $\g^i((M'_{0,\infty})^{-i-1})$. Also, $(M'_{0,\infty})^{-i} = 0$
as $M'_0$ and $M'_\infty$ are either both even or both odd. Thus $d(m) \in
  P^{-i+1}_\g \subset (M'_{0,\infty})^{-i+1}$. Because $\g$ satisfies
HR on $M'$ we have
\[
0 \ne  \langle \g^{i-1} d(m), d(m) \rangle = \langle \g^{i-1} m,
 (d' \circ d)(m) \rangle = \langle m, \g^{i}m \rangle.
\]
This contradicts \eqref{eq:orthog}.
\end{proof}

\begin{remark}
  The above proposition reduces to Proposition \ref{prop:wL} if $M$
  and $M'$ are skyscraper sheaves.
\end{remark}

\subsection{Opposite signs and the limit lemma}  \label{opp}

For the Hodge-Riemann relations to have a hope of holding one needs to place some
assumptions on the signs at 0 and $\infty$. (We have already seen a
hint of this in Example \ref{ex:constHR}. This will become clearer
in the next section, where we discuss the structure
theory of polarised $\PM^1$-sheaves.)

We say that a polarised $\PM^1$-sheaf $M$ is polarised
  with \emph{opposite signs} if:
\begin{enumerate}
\item $M_0$ and $M_\infty$ are either both even or both odd;
\item the global sections of $M$ are generated in degrees $\le 0$;
\item both $\langle-, - \rangle^0$ and $\langle -, - \rangle^\infty$
  satisfy HR;
\item if we denote by $P_0^d \subset M_0^d$ and $P_\infty^d \subset
  M_\infty^d$ the primitive subspaces, then, for all $d \le 0$, the restriction of $\langle  -,
  - \rangle^0$ to $P_0^d$ and $\langle -, - \rangle^\infty$ to
  $P_\infty^d$ are definite of opposite signs.
\end{enumerate}

Let $N$ be a free $A$-module generated in degrees $\le -2$ and
equipped with a $K$-valued non-degenerate form $\langle -, -\rangle_N
: N \times N \to K$ satisfying HR. We can build a constant $\PM^1$-sheaf out of $N$
by setting $M_0 = M_\infty = N$ and $M_{\CM^*} = N/(z)$ with $\rho_0,
\rho_\infty$ being the quotient maps. We can equip $M$ with a
polarisation by setting $\langle -, -\rangle^0 = \langle -, -\rangle_N
= -\langle -, -\rangle^\infty$. Because $N$ satisfies HR this
polarisation has opposite signs. A $\PM^1$-sheaf which is isomorphic
(isometrically for polarisations) to such an $M$ we will call
\emph{polarised constant}.

\begin{remark}
  In the following lemma the ``opposite signs'' assumption is
  crucial. It occurs in a large class of examples coming from Soergel
  bimodules (as we will explain). We do not properly understand its geometric meaning.
\end{remark}

\begin{lem} \label{lem:lim}
(``Limit lemma'') Let $M$ be a polarised $\PM^1$-sheaf
  with opposite signs. Consider $\g = (az, bz) \in Z^2$ with
  $0 < b < a$. Then $\g$ satisfies HR on $M$ for $a/b \gg 0$. Moreover
 the signs agree with the signs on $M_0$: if $m = (m_0,
  m_\infty)$ denotes a non-zero element of minimal degree $-d$ in $M_{0,\infty}$ then
  $\langle \g^d m, m \rangle$ and $\langle m_0, m_0 \rangle^0$ have
  the same sign for $a/b \gg 0$. (The map $m \mapsto m_0$ is an isomorphism in degree $-d$, as follows
from Lemma \ref{lem:P1clas}.)
\end{lem}

This lemma will be obvious later (see Lemma \ref{lem:pollim}) once
we have developed the structure theory of polarised $\PM^1$-sheaves.

In the following the assumptions on $M$ are as in Lemma \ref{lem:lim}.

\begin{cor} \label{cor:lim}
Suppose that $\g = (az, z)$ satisfies hard Lefschetz
  on $M$ for all $a \in I$, where $I \subset \RM$ is a
    connected subset which is not bounded above. Then $\g$ satisfies HR on $M$ for all 
  $a  \in I$. In particular, if $I = [1,\infty)$ then the
  global sections $M_{0,\infty}$ satisfy HR.
\end{cor}

\begin{proof} 
For any fixed $d \le 0$ the form $(x,y) \mapsto \langle x,
  \g^{-d} y \rangle$ on $M^d_{0,\infty}$ varies continuously
  in $\g$. If $\g$ satisfies hard Lefschetz for all $a \in I$ then
  these forms are non-degenerate, and the previous lemma says that for
  $a \gg 0$ these forms have signatures given by the Hodge-Riemann
  relations (see Remark \ref{rem:HRsignature}). The lemma now follows,
  as the signature of a continuous family of real
  non-degenerate forms is constant.
\end{proof}


\subsection{Structure theory of polarised $\PM^1$-sheaves} \label{sec:structure}
Throughout this section $M$ denotes a polarised $\PM^1$-sheaf. We
assume in addition that $M_0$ and $M_\infty$ satisfy HR
(with respect to the forms $\langle -, - \rangle^0$ and $\langle -, -
\rangle^\infty$).

The goal of this section is to show that $M$ admits a canonical
decomposition into simpler pieces. That is, we will see that the decomposition
in Lemma \ref{lem:P1clas} becomes canonical in the presence
of a polarisation satisfying HR at 0 and $\infty$.

\begin{lem} \label{lem:split skyscraper}
  We have a canonical decomposition
\[
M = M' \oplus N
\]
such that:
\begin{enumerate}
\item $N$ is a skyscraper: i.e. $N_\infty =
  N_{\CM^*} = 0$;
\item the induced decomposition of $M_0$ is orthogonal
for $\langle -, -\rangle^0$;
\item the induced map $M_0' \to M'_{\CM^*} = M_{\CM^*}$ is a projective cover.
\end{enumerate}
\end{lem}


Let $? \in \{0, \infty \}$. By assumption $M_?$ satisfies
hard Lefschetz. In particular it is generated in degrees $\le 0$. For $d \le 0$ let $P_?^d \subset M_?^d$ denote the
primitive subspace (see \S\ref{def:hr}) and set $P_? = \bigoplus P_?^d \subset M_?$.

\begin{proof} 
  Let $L$ denote the kernel of the composition $P_0 \into M_0
  \stackrel{\rho_0}{\longto} M_{\CM^*}$. Let $L^\perp \subset P_0$ denote the
  orthogonal to $L$ under the Lefschetz form (see
  \S\ref{def:hr}) degree by degree. By our HR assumption, 
  each Lefschetz form on $P_0$ is definite in any fixed degree. Hence
  $P_0 = L \oplus L^\perp$. This leads to a canonical
  decomposition (see Lemma \ref{lem:prim1})
  \[
  M_0 = A \otimes_\RM L \oplus A \otimes_\RM L^{\perp}.
  \]
  Hence we can write our sheaf as a direct sum $M = N \oplus
 M'$ where $N$ is the skyscraper sheaf at zero associated to $
A \otimes_\RM L$
  (i.e. $N_0 = A \otimes_\RM L$, $N_{\CM^*} = N_\infty = 0$). (1) and (2) are now
  clear. (3) follows because the composition $L^\perp \to M_0 \to
  M_{\CM^*} = M'_{\CM^*}$ is an isomorphism by construction.
\end{proof}

\begin{lem} \label{lem:split primitive}
Let $M$ be as above and assume additionally that $\rho_0 : M_0 \to
  M_{\CM^*}$ is a projective cover. We have a canonical
  decomposition
\[
M = \bigoplus_{i \le 0} M_i
\]
where each $M_i$ is isomorphic to a direct sum of constant sheaves
generated in degree $i$ (ignoring forms). Moreover the induced
decomposition of $M_0$ (resp. $M_\infty$) is orthogonal with respect
to $\langle -, -\rangle^0$ (resp. $\langle -, -\rangle^\infty$).
\end{lem}

\begin{proof}
  Under the assumptions of the lemma the induced maps
\[
P_0 \to M_{\CM^*} \from P_\infty
\]
are isomorphisms. Now the canonical decompositions
\[
M_0 = \bigoplus_{i \le 0} A \otimes_{\RM} P_0^i \qquad M_\infty = \bigoplus_{i \le 0} A\otimes_{\RM} P_\infty^i
\]
lead to the desired decomposition.
\end{proof}

\begin{remark} Two forms on a real vector space may be simultaneously
  diagonalised if one form is definite. (I thank Pavel Etingof for
  this remark.) Hence we could further decompose our
  polarised $\PM^1$-sheaf into a direct sum of polarised sheaves
  of rank 1. The decomposition of Lemma \ref{lem:split
    primitive} is enough for our needs.
\end{remark}

\subsection{Hodge-Riemann revisited} One can use the above structure
theory to give a simple criterion for HR to be satisfied.

Let $M$ be a $\PM^1$-sheaf which is polarised with opposite signs,
and whose global sections are generated in degrees $\le 0$. We
would like to know when the global sections of $M$ satisfy HR with the
same signs as those on $M_0$.

Let $M = N \oplus M'$ be the decomposition of Lemma \ref{lem:split
  skyscraper} (so $N$ is a skyscraper).  It is easy to see that
the Hodge-Riemann bilinear relations are
  always satisfied (with the correct sign) for the summand of the
  global sections coming from
  $N$. Hence we can assume that $\rho_0 : M_0 \to M_{\CM^*}$ is a
  projective cover. By Lemma \ref{lem:split primitive} we may even
  assume that $M$ is of the following form:
  \begin{enumerate}
  \item $M_{\CM^*} = V$ for some finite dimensional graded real vector
    space concentrated in fixed degree $d \le -2$;
  \item $M_0 = M_\infty = A \otimes_\RM V$;
  \item there exists symmetric forms $( -, -)^0$ and $(-,-)^\infty$ on
    $V$ which are definite of opposite signs and such that the local
    forms are given by
    \begin{gather*}
      \langle 1 \otimes v, 1 \otimes v'  \rangle^0 = (v,v')^0 z^{d} \\
      \langle 1 \otimes v, 1 \otimes v' \rangle^\infty = (v,v')^\infty
      z^{d}
    \end{gather*}
    for all $v, v' \in V$.
  \end{enumerate}

  \begin{lem} \label{lem:almost-const}
    The global sections of $M$ satisfy HR (with the same signs as
    $M_0$) if and only if the form $(-,-)^0 + (-,-)^\infty$ on $V$ is definite (of the
    same sign as $(-,-)^0$).
  \end{lem}

  \begin{remark}Informally, the global sections of a $\PM^1$-sheaf which is polarised
    with opposite signs satisfies HR if the
    ``form at 0 dominates the form at $\infty$''. This will
    be a subtle question in general!
  \end{remark}

  \begin{proof}
The graded rank of $M_{0,\infty}$ is $v^d + v^{d+2}$. To verify hard Lefschetz and HR
  it is enough to show that the form $z^{-m}\langle x, y \rangle$ on
  $M_{0,\infty}^m$ is non-degenerate for $m = d, d+2$, and that its signature is the
  same as that of $(-,-)^0$ on $M_{0,\infty}^d$, and is 0 on
  $M_{0,\infty}^{d+2}$ (see Lemma \ref{lem:signs}).

The global sections of degree $d$ are given by the diagonal $v \mapsto
(1 \otimes v, 1 \otimes v)$. The restriction of $\langle -, -\rangle$
to $M_{0,\infty}^d$ is given by
\[
\langle (1 \otimes v, 1\otimes v), (1 \otimes v', 1 \otimes v') \rangle =
((v,v')^0 + (v,v')^\infty) z^{d}.
\]
Hence the form $z^{-d}\langle x, y \rangle$ is non-degenerate and
 the Hodge-Riemann relations are satisfied in this degree (with the correct signs)  if and
only if $(v,v)^0 + (v,v)^\infty$ is non-zero and of the same sign as
$(v,v)^0$, for all $0 \ne v \in V$.

The map $(v,v') \mapsto (z \otimes v, z \otimes v')$ gives an
isomorphism between $V \oplus V$ and the global sections in degree $d
+ 2$. This isomorphism identifies the form $z^{-d-2}\langle x, y
\rangle$ on $M_{0,\infty}^{d+2}$
with the direct sum of the forms $(-,-)^0$ and $(-,-)^\infty$ on $V
\oplus V$. So the non-degeneracy and HR relations in
this degree follow automatically from our assumption that $(-,-)^0$
and $(-,-)^\infty$ are definite of opposite signs.
  \end{proof}

The following lemma is an equivalent formulation of Lemma \ref{lem:lim}:

\begin{lem} \label{lem:pollim}
  Suppose that $M$ is a polarised $\PM^1$-sheaf with opposite
  signs. For non-zero $a > b > 0$  consider the rescaled polarisation $a\langle -, - \rangle^0$ on
  $M_0$ and $b \langle -, - \rangle^\infty$ on $M_\infty$. Then if
  $a/b \gg 0$, the global sections $M_{0,\infty}$ satisfy HR
with
  signs agreeing with those of $M_0$ (see Lemma \ref{lem:lim}).
\end{lem}

\begin{proof} By the above structure theory we have an orthogonal
  decomposition of $M$ into a direct sum of skyscraper sheaves
  (satisfying HR) and constant sheaves of the form of Lemma
  \ref{lem:almost-const}. In the notation of Lemma
    \ref{lem:almost-const} 
$a(-,-)^0 + b(-,-)^\infty$ is definite of the same sign as
    $(-,-)^0$ if $a/b \gg 0$. Now the result follows from Lemma \ref{lem:almost-const}.
\end{proof}


Recall the notion of a polarised constant sheaf $M$ (see
  \S\ref{opp}).

\begin{lem}(``HR in constant case'') \label{lem:const} If $M$ is
  polarised constant then  $M$ satisfies HR.
\end{lem}

\begin{proof}
By Lemma \ref{lem:prim1} we can assume that $M_0, M_\infty$ are
generated in one degree.
Let $\g = (az, bz) \in Z^2_\amp$. We have to verify that the form $\langle \g^{-m}
x, y \rangle$ on $M_{0,\infty}^m$ is non-degenerate with the correct
signature for $m \le 0$. As in Lemma \ref{lem:almost-const} we can
reduce to the case of the minimal non-zero degree $d$, in which case
we are asking whether
$a^{-d}\langle-,-\rangle^0 + b^{-d}\langle-,-\rangle^\infty$ is definite of the same sign as
$\langle-,-\rangle^0$. However this is the case because $\langle-,-\rangle^\infty
= -\langle-,-\rangle^0$ ($M$ is assumed polarised constant) and $a > b > 0$.
\end{proof}

\section{Soergel bimodule background} \label{sbimb}

\subsection{Bimodules} \label{bim}
Let $R$ be the regular functions on $\hg$, as above.
We will work mostly inside the category $R\bim$ of graded
$R$-bimodules (with degree zero morphisms) which are finitely
generated as both left and right $R$-modules. Given
$M, N \in R\bim$ we write $\Hom^{\bullet}(M,N) = \bigoplus_{i \in \ZM}
\Hom(M,N[i])$  for the graded vector space of morphisms of all degree
(and similarly for other graded objects, for example graded left
$R$-modules).

The category $R\bim$ is a monoidal category via
tensor product of bimodules. We denote the monoidal structure simply
by juxtaposition: given $M, N \in R\bim$ their tensor product is
\[
MN := M \otimes_R N.
\]
Given elements $m \in M$, $n \in N$ we abbreviate $mn := m \otimes n
\in MN$. We also employ this notation for morphisms: given $f : M \to
M'$, $g : N \to N'$ the (horizontal) tensor product of these two
morphisms is written $fg : MN \to M'N'$. Following standard practice we
will often use the symbol denoting an object to 
also denote its identity morphism. For example $fN$ denotes the
morphism $f \id_N : MN \to M'N$. Given $r \in M$ the morphism obtained
by left (resp. right) multiplication by $r$ is denoted $rM$ (resp. $Mr$).

Given an $R$-bimodule $M$ its dual is
\[
\DM M := \Hom^{\bullet}_{R -}(M, R)
\]
(homomorphisms of all degrees of left $R$-modules). Then $\DM M$ is a
graded $R$-bimodule via $(r \cdot f)(m) = f(rm)$ and $(f \cdot r)(m) =
f(mr)$. This definition is only sensible for bimodules which are free
and finitely generated as graded left $R$-modules. This will always be the
case below. For such bimodules the natural morphism $M \to
\DM(\DM(M))$ is an isomorphism.

\subsection{Soergel bimodules}  \label{sbim}
For background on Soergel bimodules see \cite{Soe3, EK, EW, EWSC}
  and the references therein.

We write $\BC$ for the category of Soergel bimodules. By
definition $\BC$ is the full graded additive monoidal Karoubian subcategory of $R\bim$
generated by the bimodules
\[
B(s) := R \otimes_{R^s} R[1].
\]
Given an expression $\un{w} := s_1s_2 \dots s_m$ we denote the
corresponding Bott-Samelson bimodule by
\[
B(\un{w}) := B(s_1) B(s_2) \dots B(s_m).
\]
If $\un{w}$ is reduced then $B(\un{w})$ contains a unique summand
which is not isomorphic to a shift of a summand of any Bott-Samelson
bimodule $B(\un{w}')$ for a shorter expression $\un{w}'$. We denote
(the isomorphism class of) this bimodule by $B(w)$. Then the set
$\{ B(w) \; | \; w \in W \}$ give representatives for the isomorphism
classes of indecomposable
self-dual Soergel bimodules, and any indecomposable bimodule is isomorphic to
$B(w)[m]$ for some $w \in W$ and $m \in \ZM$.

In this paper we arbitrarily choose to consider Soergel bimodules
predominantly as left modules.

\begin{warning}
  This emphasis on left over right is the opposite to the choice made in
  \cite{EW}. It simplifies the notation a little in what follows. We
  have tried to
  include warnings like this one when the conventions of the current
  paper differ from those of \cite{EW}.
\end{warning}

We define some elements and simple morphisms between Soergel bimodules
that will play an important role in this paper. Consider the elements
\[
c_\id := 1 \otimes 1 \in B(s)^{-1} \text{ and } c_s := \frac{1}{2}
(\a_s \otimes 1 + 1 \otimes \a_s) \in B(s)^{1}.
\]
These are easily seen to give a basis for $B(s)$ as a left or right
$R$-module. One checks easily that $c_s r = r c_s$ for $r \in R$. Define the maps:
\[
m : B(s) \to R[1] : f \otimes g \mapsto fg \text{ and } \mu : R[-1] \to
B(s) : f \mapsto f c_s.
\]
(These are the units and counits (``dot'' maps) of a Frobenius algebra structure on
$B(s)$, see \cite{EK,EWSC}.)  We have the ``polynomial sliding relation'' which for $\l \in \hg^*$
takes the form
\begin{equation}
  \label{slide}
  B(s) \l = s(\l)B(s) + \langle \l, \alpha_s^\vee \rangle (\mu \circ m).
\end{equation}

\subsection{Support, stalk, costalk} \label{sec:stalkcostalk}
Any $M \in R\bim$ can be regarded
as a coherent sheaf on $\hg \times \hg$ (remember that $R$ is
commutative, so $R$-bimodules are the same thing as $R \otimes
R$-modules). For $x \in W$ consider its ``twisted graph'':
\[
\Gr_x := \{ (x\l, \l) \; | \; \l \in \hg \} \subset \hg \times \hg.
\]
One may identify the regular functions on $\Gr_x$ with the bimodule
$R(x)$ which is free or rank 1 as a left $R$-module, and has right action given
by
\[
b \cdot r = x(r)b
\]
for $b \in R(x)$ and $r \in R$.

Given any subset $X \subset W$ we set $\Gr_X := \cup_{x \in X}
\Gr_x$. Given a subset $X \subset W$ we write $B_X$
(resp. $B_X^!$) for the stalk (resp. costalk, i.e. sections with support) of $B$
along $\Gr_X$. We write $B_x$ instead of $B_{\{x\}}$ and $B_x^!$
  instead of $B_{\{x\}}^!$. We have
\[
B_x^! = \Hom^{\bullet}(R(x), B) \text{ and } B_x =
R(x) \otimes_{R \otimes R} B
\]
(where in the second equality we regard $R(x)$ and $B$ as graded
$R\otimes R$-modules). Given $x, y \in W$ we write $R(x,y)$ for the regular functions on
$\Gr_x \cup \Gr_y$ and $B_{x,y}$ for $B_{\{x,y\}}$. We have $B_{x,y} =
B \otimes_{R\otimes R} R(x,y)$.

  \begin{remark}
    The modules $B_X$ and $B_X^!$ are denoted $\Gamma^X B$ and
    $\Gamma_X B$ in \cite{Soe3}.
  \end{remark}

\begin{warning}
  Stalks and costalks appear so frequently in the present work that
  we decided to denote them $B_x$ and $B_x^!$. Let us emphasise that
  the indecomposable self-dual bimodule parametrised by $y \in W$ will
  be denoted $B(y)$ in this paper (and not $B_y$ as in \cite{Soe3, EW}). We hope that this does not cause confusion for the reader.
\end{warning}

For any Soergel bimodule $B$ the stalks and costalks $B_x$,
  $B_x^!$ are free as left $R$-modules \cite[Theorem 5.15]{Soe3} 
and we have canonical inclusions and projections
\[
B^!_x \hookrightarrow B \quad \text{and} \quad B \onto B_x
\]
which split when regarded as morphisms of left $R$-modules (see the
proof of \cite[Proposition 6.4]{Soe3}).
Recall that we write $Q = R[1/\Phi]$ for the localisation of $R$ at all roots.

Taking the direct sum over the canonical maps we obtain injections
\[
\bigoplus_{w \in W} B_w^! \hookrightarrow B \hookrightarrow
\bigoplus_{w \in W} B_w
\]
and both maps become isomorphisms after applying $Q  \otimes_R  (-)
$.  (This is not difficult to check for
Bott-Samelson bimodules, from which the general case follows.) In
particular the composition
\[
i_x :  B_x^!  \to  B_x 
\]
is an injection, which becomes an isomorphism after tensoring with
$Q$.

In what follows it will be convenient to consider the injection (an isomorphism over $Q$)
\begin{equation} \label{eq:BQ}
i: B \into \bigoplus_{w \in W} B_w \quad \quad b \mapsto (b_w).
\end{equation}
(Finitely many $B_w$ are non-zero.)

\subsection{Polarisations} \label{sec:pol}

An \emph{invariant form} on a Soergel bimodule $B$ means a
symmetric graded bilinear form
\[
\langle -, - \rangle : B \times B \to R
\]
such that $\langle rb, b'\rangle =
\langle b, rb' \rangle = r\langle b, b' \rangle$ and $\langle br, b'
\rangle = \langle b, b'r \rangle$ for all $b, b' \in B$ and $r \in
R$ (note the left/right asymmetry).

\begin{warning} This does not agree with the terminology ``invariant
  form'' in \cite{EW}, where the roles of the left and right action are
  interchanged.\end{warning}

An invariant form on a Soergel bimodule $B$ is \emph{non-degenerate} if it
induces an isomorphism $B \simto \DM B$. A \emph{polarisation} of a
Soergel bimodule $B$ is a non-degenerate invariant form $\langle -, -
\rangle_B$ on $B$. Throughout a \emph{polarised Soergel bimodule} will mean a Soergel
bimodule $B$ together with a fixed non-degenerate invariant form
$\langle -, - \rangle_B$. We will denote a polarised Soergel bimodule
by $(B, \langle -, - \rangle_B)$ or simply $B$ (in which case the form
is implicit).

Let $\un{w}$ be an expression. The set
\[
\{ c_\pi := c_{u_1} c_{u_2} \dots c_{u_m} \; | \; \un{\pi} = u_1 \dots u_m \text{
  a subexpression of } \un{w} \}
\]
gives a basis of $B(\un{w})$ as a free left $R$-module. We define the
\emph{intersection form} on $B(\un{w})$ to be
\[
\langle f, g \rangle_{B(\un{w})} = \Tr(fg)
\]
where $\Tr(fg)$ denotes the coefficient of $c_{\un{w}}$ in the above
basis, and $fg$ denotes the product of $f$ and $g$ in $B(\un{w})$ (a
ring). Then $\langle -, - \rangle_{B(\un{w})}$ is a non-degenerate
invariant form on $B(\un{w})$ (see \cite[\S 3.4]{EW} and \cite[Lemma 3.8]{EW}, remembering to
switch left and right actions). Unless we state explicitly otherwise,
we will always regard Bott-Samelson bimodules as polarised with
respect to their intersection forms.

An important case below will be given by the intersection form on
$B(s)$. In the left basis $\{ c_\id, c_s \}$ of \S\ref{sbim} we have:
\begin{equation} \label{eq:B(s)form}
\langle c_\id, c_\id \rangle = 0, \langle c_\id, c_s \rangle = \langle
c_s, c_\id \rangle = 1, \langle c_s, c_s \rangle = \alpha_s.
\end{equation}

Given two Soergel bimodules $B_1$ and $B_2$ equipped with invariant
forms $\langle - , -\rangle_{B_1}$ and $\langle -, - \rangle_{B_2}$ it
is easy to check that we get an invariant form on $B_1 B_2$ via
\[
\langle b_1 b_2, b'_1 b'_2 \rangle_{B_1 B_2} = \langle b_1 \langle b_2, b'_2
\rangle_{B_2}, b_1' \rangle_{B_1} = \langle b_1, b_1' \langle b_2, b'_2 \rangle_{B_2} \rangle_{B_1}.
\]
One may also check that if $\langle -, - \rangle_{B_1}$ and $\langle -, -
  \rangle_{B_2}$ are non-degenerate, then so is $\langle -, -
  \rangle_{B_1B_2}$. (This is clear after choosing bases and
  dual bases for $B_1$ and $B_2$.) In particular, if $B_1$ and $B_2$ are polarised,
  then so is $B_1B_2$.

\begin{remark} This construction is associative in an obvious
  sense. One may check that it returns the intersection form on a
  Bott-Samelson bimodule, starting from the intersection form on each
  of the $B(s)$ factors.
\end{remark}


\subsection{Positive polarisations} \label{sec:pp}

Recall that a Soergel bimodule $B$ is \emph{perverse} if $B =
\bigoplus B(y)^{\oplus m_y}$ for some $m_y \in \ZM_{\ge 0}$. (That is, $B$ is
isomorphic to a direct sum of indecomposable self-dual Soergel bimodules without
shifts.)

Recall that Soergel's conjecture combined with Soergel's hom formula
\cite[Theorem 5.15]{Soe3} implies that
\begin{equation} \label{eq:vanish}
\Hom(B(x),B(y)) = \begin{cases} \RM & \text{if $x = y$,} \\ 0 &
  \text{otherwise}. \end{cases}
\end{equation}
Hence if $B$ is any perverse Soergel bimodule we have a
canonical ``isotypic'' decomposition
\begin{equation} \label{eq:can}
B = \bigoplus V(z) \otimes_\RM B(z)
\end{equation}
for real (degree zero) vector spaces $V(z)$.  The following important
fact will be used repeatedly in what follows
(it also played a key role in \cite{EW}):

\begin{lem} \label{lem:orthog} The decomposition \eqref{eq:can} is orthogonal for
  any invariant form on $B$.
\end{lem}

\begin{proof}
  An invariant form yields a morphism $B \to \DM B \cong B$ 
which must respect \eqref{eq:can} by \eqref{eq:vanish}.
\end{proof}

\begin{remark} Similar arguments show that giving an invariant form on $B$ is the same thing as
  giving a symmetric form on each $V(z)$, once one has fixed
  the intersection form on each $B(z)$.
\end{remark}

In particular, a polarisation of $B(z)$ induces an
  isomorphism $B(z) \simto \DM B(z) = B(z)$, and hence is unique up to
  a scalar. As in \cite{EW} we choose
for every $y \in W$ an embedding of $B(y)$ as a summand in
$B(\un{y})$ for some reduced expression $\un{y}$ of $y$. Restricting
the intersection form on $B(\un{y})$ yields a polarisation
of $B(y)$. (The non-degenerate intersection form gives an isomorphism
  $\phi : B(\un{y}) \simto \DM B(\un{y})$. As the summand $B(y) \subset B(\un{y})$
  is unique up to isomorphism, $\phi$ restricts to an isomorphism
  $\phi: B(y) \simto \DM B(y)$. Hence the
  restriction to $B(y)$ is non-degenerate.) We call this polarisation the
intersection form (it is well-defined up to a positive scalar).

Let $(B, \langle -, -\rangle)$ be a polarised Soergel bimodule. We say
that $B$ is \emph{positively polarised} if:
\begin{enumerate}
\item $B$ is perverse and vanishes in even or odd degree;
\item if we fix a decomposition as in \eqref{eq:can} and let $z \in W$
  be maximal such that $m_z \ne 0$ then the induced form on each
  $V(y)$ is $(-1)^{(\ell(z)-\ell(y))/2}$ times a positive definite form.
\end{enumerate}
(Our assumption that $B$ vanishes in even or odd degree forces all
elements of $\{ \ell(y) \; | \; m_y \ne 0 \}$ to have the same parity, and hence
$(\ell(z)-\ell(y))/2$ in (2) makes sense.)

\begin{remark} \label{rem:bbs}
Suppose $y \in W$ and $s \in S$ with $ys > y$. Then
  $B(y)B(s)$ is perverse (as follows from Soergel's conjecture), and
  has a natural form induced from the intersection forms on $B(y)$ and
  $B(s)$ (see
  \S\ref{sec:pol}). This yields a positive polarisation \cite[Proposition 6.12]{EW}.
\end{remark}

\subsection{Adjoints}

Let $B$, $B'$ be polarised Soergel bimodules. Given a map $f : B \to B'[m]$ (i.e. $f$ is a degree
$m$ map from $B$ to $B'$) we denote by
$f^* : B' \to B[ m]$ the adjoint map. It is uniquely determined by the property
\[
\langle b, f^*(b') \rangle_B = \langle f(b), b' \rangle_{B'}
\]
for all $b \in B$, $b' \in B'$. In particular $f = (f^*)^*$.

Recall the ``dot'' maps $m : B(s) \to R[1]$ and $\mu : R \to
B(s)[1]$ from \S\ref{sbim}. An easy calculation shows that (with respect to the
intersection forms on $B(s)$ and $R$):
\begin{equation}
  \label{eq:duals}
m = \mu^* .
\end{equation}

Let $B_1$ and $B_2$ be two polarised Soergel bimodules. Then 
if $f_1 : B_1 \to B_1'[i]$ and $f_2 :
B_2 \to B_2'[i']$ are morphisms then
\begin{equation} \label{eq:prodadjoint}
(f_1 f_2)^* = f_1^* f_2^* :
B_1'B_2' \to B_1 B_2[i+i'].
\end{equation}

\subsection{Local forms}

Now suppose that $B$ is polarised via
\[
\langle - , - \rangle : B \times B \to R.
\]
By extension of scalars we obtain a form
\[
\langle - , - \rangle_Q : Q \otimes_R B\times Q \otimes_R B\to Q.
\]

\begin{lem} \label{lem:orth}
  The form $\langle -, - \rangle_Q$ is orthogonal with respect to
  the decomposition in \eqref{eq:BQ}.
\end{lem}

\begin{proof}
  Suppose that $b \in B_x$ and $b' \in B_{y}$. Then, for all $r \in
  R$ we have
\[
r \langle b, b' \rangle = \langle b, r b'  \rangle = \langle b, b' y^{-1}(r)
\rangle = \langle by^{-1}(r) , b' \rangle = \langle xy^{-1}(r) b , b' \rangle
= xy^{-1} (r) \langle b, b' \rangle.
\]
Hence if $\langle b, b' \rangle \ne 0$ then $x = y$ (remember that $R$ is an
integral domain and $W \to Aut(R)$ is faithful).
\end{proof}

\begin{defi}
  We write $\langle -, - \rangle_B^w$ (or $\langle -, - \rangle^w$ if
  the context is clear) for the induced $Q$-valued 
  form on $B_w$ and call it the \emph{local intersection form}.
\end{defi}

\begin{remark}
  This local intersection form is not the same as the local
  intersection form considered in \cite{EW}. In fact, the local
  intersection forms considered in \cite{EW} may be ``embedded'' into
  those above. We will not discuss this here, but see \S\ref{sec:sc}.
\end{remark}

The following proposition summarises the key properties of the local
intersection form:

\begin{prop} 
  \begin{enumerate}
  \item For all $b, b' \in B$ we have
\[
\langle b, b' \rangle = \sum_{w \in W} \langle b_w, b'_w \rangle^w.
\]
\item $\langle -, - \rangle^w$ induces a non-degenerate graded form on
  $Q  \otimes_R B_w$.
\item $B_w, B_w^! \subset  Q \otimes_R B_w$ are dual lattices with
  respect to $\langle -, - \rangle^w$.
  \end{enumerate}
\end{prop}

\begin{proof}
  (1) and (2) follow from the definitions. For (3) note that
  by (1) and Lemma \ref{lem:orth}, $\langle B_w^!, B_w \rangle^w = \langle B_w^!, B \rangle  \subset R$.
Hence our non-degenerate form gives an injection:
\[
B_w^! \to (B_w)^*.
\]
Now if we compare graded ranks (given by Soergel's hom formula) we see
that our map is an isomorphism.
\end{proof}

\subsection{Local induced forms} Throughout this section we fix a Soergel
bimodule $B$. The goal is to relate two forms on the Soergel bimodule $BB(s)$.

\begin{prop} \label{prop:BB(s)}
For any Soergel bimodule $B$, $x \in W$ and $s \in S$ we have a canonical
  identification $(BB(s))_x = B_{x,xs}[1]$ (as left $R$-modules).
\end{prop}

\begin{proof}For the proof let us work in the category of $R\otimes R$-modules,
viewing all $R$-bimodules as $R \otimes R$-modules. We have (all
unspecified tensor products are over $\RM$):
  \begin{align*}
(BB(s)) \otimes_{R \otimes R} R(x)[-1]
&= B \otimes_{R \otimes R} ( R
\otimes R \otimes_{R \otimes R^s} R \otimes R) \otimes_{R \otimes R}
R(x) \\
&= B \otimes_{R \otimes R} ( R \otimes R \otimes_{R \otimes R^s} R(x))
\\
&= B \otimes_{R \otimes R} (R(x)B(s))[-1]  \\
&= B \otimes_{R \otimes R}
R(x,xs) \\
& = B_{x,xs}.
\end{align*}
(We have used the isomorphism $R(x)B(s)[-1] = R(x,xs)$. This follows
easily from $B(s)[-1] = R(\id, s)$, which can be checked by hand.)
The proposition now follows.
\end{proof}

\begin{remark}
  Recall that the invariant ring $R^s$ is the ring of regular functions on the
  quotient $\hg / \langle s \rangle$. In the
  language of coherent sheaves the functor of tensoring
  on the right with $B(s)$ is isomorphic to $\pi^* \pi_*[1]$ where $\pi :
  \hg \times \hg \to \hg \times \hg/\langle s \rangle$ is the quotient
  map. The above proof is an algebraic translation of simple facts
  about the effect of pushforward and pullback on stalks.
\end{remark}

\begin{lem}
  The natural map $B_{x,xs} \to B_x \oplus B_{xs}$ is injective.
\end{lem}

\begin{proof}
  The map in question becomes an isomorphism after applying $Q
  \otimes_R (-)$. Hence it is enough to show that $B_{x,xs}$ is
  torsion free as a left $R$-module. However this follows from the
  previous proposition, as the stalks and costalks of Soergel
  bimodules are free as left $R$-modules \cite[Theorem 5.15]{Soe3}.
\end{proof}

\begin{lem} \label{lem:indmult}
  Let $f, g, h \in R$. For any $x \in W$ we have a commutative diagram
\[
\xymatrix{
B B(s) \ar@{>>}[r] \ar[d]_{f B g B(s) h} & (B B(s))_x \ar[r]^{\sim} \ar[d]
& B_{x,xs} \ar[d] \lhook\mkern-7mu \ar[r] 
& \ar[d]_{\g\cdot}  B_{x} \oplus B_{xs}\\
B B(s) \ar@{>>}[r] & (B B(s))_x \ar[r]^{\sim} & B_{x,xs} \lhook\mkern-7mu \ar[r] &
B_{x} \oplus B_{xs} 
}
\]
where $\g = (fx(g)x(h), fx(g)xs(h)) \in R \oplus R$. (The two isomorphisms
are those of Proposition \ref{prop:BB(s)}. All other horizontal maps are canonical.)
\end{lem}

\begin{proof}
  This follows easily by chasing $f b g b' h \otimes 1$ through the 
  identifications in the proof of Proposition \ref{prop:BB(s)}.
\end{proof}

Now let us assume that $B$ is polarised by $\langle -, -
\rangle_B$. Then $B_{x,xs}$ carries a form induced by the
sum of the two local intersection forms on $B_x$ and $B_{xs}$ under
the inclusion $B_{x,xs} \into B_x \oplus B_{xs}$. On the
other hand $(BB(s))_x$ carries a local intersection form (coming from
the induced form on $BB(s)$). The following proposition relates these
forms:

\begin{prop} \label{prop:indint}
  Let $B$ be a polarised Soergel bimodule. Then under the
  identification
\[
(BB(s))_x = B_{x,xs}[1] \subset (B_x \oplus B_{xs})[1]
\]
of Proposition \ref{prop:BB(s)}
  we have
\[
\langle -, - \rangle_{BB(s)}^x = \frac{1}{x\alpha_s}( \langle -, -
\rangle_B^x + \langle -, -\rangle_B^{xs}).
\]
\end{prop}

\begin{proof}
Recall that for a general Soergel bimodule $B'$ we denote the map $B'
\to \bigoplus B'_x$ by $b \mapsto (b_x)$. For the course of the proof
let $j_x$ denote the composition
\[
j_x : (BB(s))_x \simto B_{x,xs}[1] \into (B_x \oplus B_{xs})[1]
\]
where the first map is the identification of Proposition
\ref{prop:BB(s)} and the second map is the inclusion.
We have (as follows from a simple calculation):
\begin{align}
j_x(bc_\id) &= (b_x, b_{xs}) \label{eq:w1},\\
j_x(bc_s ) & = (b_x\alpha_s , 0) = x \a_s (b_x, 0).\label{eq:w2}
\end{align}
For the course of the proof let us write $\langle -, -
\rangle_{ind}^x$ for the form displayed on the right hand side in the
proposition. We want
to show $\langle -, - \rangle_{BB(s)}^x = \langle -, -
\rangle_{ind}^x$. In order to check this it is enough to show that
if we define a form on $BB(s)$ via
\[
\langle b, b' \rangle_{ind} := \sum_{x \in W} \langle j_x(b), j_x(b')
\rangle_{ind}^x
\]
then we have
\[
\langle -, - \rangle_{ind} = \langle -, - \rangle_{BB(s)}.
\]
We do this by checking the defining properties of the induced form:
\begin{gather}
  \langle bc_{\id},    b'c_{\id} \rangle_{BB(s)} = 0 \label{eq:ind1} \\
\langle bc_s , b'  c_\id \rangle_{BB(s)} = \langle b c_\id, b' c_s \rangle_{BB(s)} =
\langle b, b' \rangle_{B} \label{eq:ind2}\\
  \langle b c_s, b' c_s \rangle_{BB(s)} = \langle b, b' \alpha_s \rangle_{B}
  =\langle b \a_s , b'  \rangle_{B}\label{eq:ind3}
\end{gather}
(these formulas follow from \eqref{eq:B(s)form} and the definition of the induced form).

Firstly, by \eqref{eq:w2} we have (for any $b, b' \in B$):
\begin{align*}
  \langle bc_\id , b' c_\id \rangle & = \sum_{x \in X} \langle
  j_x(bc_\id), j_x(b'c_\id) \rangle^x_{ind} \\
& = \sum_{x \in W} \langle (b_x, b_{xs}), (b_x, b_{xs}')
\rangle^x_{ind} \\
& = \sum_{x \in W} \frac{1}{x\alpha_s} ( \langle b_x, b_x' \rangle_B^x + \langle b_{xs},
b_{xs}' \rangle_B^{xs} ) \\
& = 0
\end{align*}
(The last line follows by breaking the sum into two pieces
corresponding to $xs > x$ and $xs < x$ and using
that $s(\a_s) = -\a_s$.) This gives \eqref{eq:ind1}.

For \eqref{eq:ind2} we have:
\begin{align*}
\langle b c_s , b' c_{\id} \rangle_{ind} & = 
\sum_{x \in W} \langle j_x(bc_s), j_x(b'c_{\id}) \rangle^x_{ind} \\
& = \sum_{x \in W} \frac{1}{x\a_s}  \langle b_x \alpha_s, b'_x \rangle_B^x
\\
& = \sum_{x \in W} \langle b_x, b_x' \rangle_B^x \\
& = \langle b, b' \rangle_B.
\end{align*}
(We use that $b''\a_s = x(\a_s)b''$ for all $b'' \in B_x$.)
An almost identical calculation shows that $\langle b c_{\id} , b' c_s
\rangle_{ind} = \langle b, b' \rangle_B$.

For \eqref{eq:ind3} we have
\begin{align*}
\langle bc_s , b'c_s \rangle_{ind} & = \sum_{x \in W} \frac{1}{x\a_s}\langle b_x
\a_s, b'_x \a_s \rangle_{ind}^x = \sum_{x \in W} \langle b_x, b_x' \a_s
\rangle_B^{x} = \langle b, b'\a_s  \rangle_B.
\end{align*}
Hence $\langle -, - \rangle_{ind} = \langle -, - \rangle_{BB(s)}$ and
the proposition follows.
\end{proof}

\subsection{Local intersection forms and the equivariant multiplicity}
\label{sec:em}

Recall the nil Hecke ring from \S\ref{sec:nilhecke}. For any
expression $\un{y} = u \dots ts$ define $e_{x,\un{y}} \in Q$ via
\[
D_{\un{y}} := D_u \dots D_tD_s = \sum e_{x,\un{y}} \delta_x.
\]
Then $e_{x,\un{y}} =
e_{x,y}$ if $\un{y}$ is reduced and is zero
otherwise (see \S\ref{sec:nilhecke}). Let $c_{x,\un{y}}$ denote the image of $1
\otimes 1 \otimes \dots \otimes 1$ in $B(\un{y})_x$ and let $\langle
-, - \rangle_{B(\un{y})}^x$ denote the local intersection form on $B(\un{y})_x$.

\begin{lem} We have
  $\langle c_{x,\un{y}}, c_{x,\un{y}} \rangle_{B(\un{y})}^x = e_{x,\un{y}}$.
\end{lem}

\begin{proof}
We prove the lemma by induction on the length of $\un{y}$, with the
case of the empty sequence being straightforward. Let $\un{y}' := u \dots t$
denote the expression obtained from $\un{y}$ by deleting the final
$s$. Under the identifications and injection
\[
B(\un{y})_x = (B(y')B(s))_x = B(y')_{x,xs} \into B(y')_x \oplus B(y')_{xs}
\]
one checks that $c_{x,\un{y}}$ maps to $(c_{x,\un{y'}}, c_{xs,\un{y'}})$. By
Proposition \ref{prop:indint} and induction we have
\begin{align*}
\langle c_{x,\un{y}}, c_{x,\un{y}} \rangle_{B(\un{y})}^x & = 
\frac{1}{x\a_s}( 
\langle c_{x,\un{y'}}, c_{x,\un{y'}} \rangle_{B(\un{y'})}^x +
\langle c_{xs,\un{y'}}, c_{xs,\un{y'}} \rangle_{B(\un{y'})}^{xs} ) \\
& = \frac{1}{x\a_s}( e_{x,\un{y}'} + e_{xs,\un{y}'}) = e_{x,\un{y}}
\end{align*}
where the last equality follows by expanding $D_{\un{y}'}D_s$.
\end{proof}


Recall that for all $y$ we have fixed a realisation of $B(y)$ as a
summand of $B(\un{y})$, for some reduced expression $\un{y}$ for
$y$. Let us denote by $c_\bot$ (resp. $c_{x,y}$) the image of $1
\otimes \dots \otimes 1$ in $B(y)$ (resp. $B(y)_x$). Because
$B(y)^{-\ell(y)}$ is one-dimensional, $c_\bot$ and $c_{x,y}$ are 
well-defined up to a non-zero scalar.

\begin{thm} \label{thm:em}
We have $\langle c_{x,y}, c_{x,y} \rangle_{B(y)}^x =
  \gamma e_{x,y}$ for some $\gamma \in \RM_{>0}$.
\end{thm}

\begin{proof} Let us denote by $i : B(y) \into B(\un{y})$ our fixed
  realisation of $B(y)$ as a summand of
  $B(\un{y})$. Recall that
  the intersection form on $B(y)$ is defined as the restriction of
the intersection form on $B(\un{y})$. Hence we need to
calculate $\langle i(c_{x,y}), i(c_{x,y})
\rangle_{B(\un{y})}^x$. 
However $B(y)_x$ and $B(\un{y})_x$ are both
generated in degrees $\ge -\ell(y)$ and are of dimension 1 in degree
$-\ell(y)$. It follows that $i(c_{x,y}) = c_{x,\un{y}}$ and the theorem
  follows from the previous lemma.
\end{proof} 

\begin{remark} \label{rem:scalar}
  Actually one may prove the existence of elements $c_\bot \in
  B(y)^{-\ell(y)}$ and $c_{x,y} \in B(y)_x^{-\ell(y)}$ which are
  canonical up to sign, once one has fixed a positive polarisation on
  $B(y)$. With this choice the scalar factor $\g$ in the
  above theorem disappears. One proceeds as follows: for any reduced
  expression $\un{y}$ the positive integer $N$ appearing in the the
  proof of \cite[Lemma 3.10]{EW} is easily seen to depend only on
  $y$. Now $c_\bot$ (and hence $c_{x,y}$) is fixed up to sign by requiring that $\langle
  c_\bot, \rho^{\ell(y)} c_\bot \rangle_{B(y)} = N$.
\end{remark}

\subsection{Soergel bimodules and $\PM^1$-sheaves} \label{sec:SP1}
Let $B$ be a
Soergel bimodule and fix $x \in W$ and $s \in S$ with $x < xs$.

To this data we may associate a $\PM^1$-sheaf $M(B,x,xs)$ as follows (we set $M :=
M(B,x,xs)$ to simplify notation, why we obtain a $\PM^1$-sheaf
will be explained later):
\begin{enumerate}
\item $M_0 = A \otimes_R B_x[1]$, $M_\infty = A \otimes_R B_{xs}[1]$;
\item $M_{\CM^*}$ is defined as the push-out of (left, graded) $A$-modules:
\[
\xymatrix{ A \otimes_R B_{x,xs}[1] \ar[r] \ar[d] & A \otimes_R B_{xs}[1] \ar[d] \\
A \otimes_R B_{x}[1] \ar[r] & M_{\CM^*}}
\]
\item $\rho_0 : M_0 \to M_{\CM^*}$ and $\rho_\infty : M_\infty \to M_{\CM^*}$ are the maps occurring in the above push-out diagram.
\end{enumerate}
Now $B_{x,xs} \to B_x \oplus B_{xs}$ is an injective map of free
$R$-modules which is an isomorphism over $Q$ and hence
$A \otimes_R B_{x,xs} \to A \otimes_R B_x \oplus 
A \otimes_R B_{xs}$ is injective. Hence we have a canonical isomorphism
\begin{equation} \label{eq:sec}
A \otimes_R B_{x,xs}[1] = M_{0,\infty}.
\end{equation}
Also, as the
inclusion $B_{x,xs} \to B_x \oplus B_{xs}$ becomes an isomorphism
after inverting $x(\alpha_s)$, $M_{\CM^*}$ is
annihilated by $0 \ne \sigma(x(\alpha_s))$. Hence we indeed have a sheaf on the
moment graph of $\PM^1$.

\begin{prop} \label{isP1}
  $M = M(B,x,xs)$ is a $\PM^1$-sheaf.
\end{prop}

\begin{proof} Deferred until the next section. \end{proof}

\begin{remark} Alternatively, one can deduce Proposition \ref{isP1} from results
  of Fiebig (indeed it was Fiebig's work that led the author to consider
  $\PM^1$-sheaves). In \cite[Proposition
  7.1]{Fiebig} Fiebig shows that one may obtain $B$ as the global
  sections of a sheaf $\BS$ on the moment graph of $W$ (we refer the
  reader to \cite{Fiebig} for unexplained terminology). The
  $\PM^1$-sheaf defined above is obtained by restricting $\BS$ to the
directed  subgraph $x \to xs$ and applying $A\otimes_{R} (-)$. It now follows
  from 
  \cite[Proposition 7.4]{Fiebig} that we obtain a $\PM^1$-sheaf.
\end{remark}

If $B$ carries a polarisation then we can equip $M$ with a
polarisation via:
\begin{enumerate}
\item $\frac{1}{\sigma(x(\alpha_s))} \sigma(\langle -, -\rangle^x)$ on $M_0$;
\item $\frac{1}{\sigma(x(\alpha_s))} \sigma(\langle -, -\rangle^{xs}) =
  - \frac{1}{\sigma(xs(\alpha_s))} \sigma(\langle -, -\rangle^{xs})$ on $M_\infty$.
\end{enumerate}
Combining Proposition \ref{prop:BB(s)} and \eqref{eq:sec} we have an
identification
\begin{equation}
  \label{eq:globsec}
  M_{0,\infty} = A \otimes_R (B B(s))_x
\end{equation}
and by Proposition \ref{prop:indint} we conclude that

\begin{lem} \label{lem:isom}
  \eqref{eq:globsec} is an isometry.
\end{lem}

\subsection{Proof of Proposition \ref{isP1}} We keep the notation of
the previous section. Our goal is to show that
$M := M(B,x,xs)$ is  a
$\PM^1$-sheaf. This is immediate from the following proposition:

\begin{prop} \label{Bxxs}
  Given a Soergel bimodule $B$
 and $x \in W$, $s \in S$ with $x < xs$, $B_{x,xs}$ is isomorphic to a
  direct sum of shifts of $R(x,xs)$ and $R(x)$.
\end{prop}

Before giving the proof we need some terminology. We say that a graded
$R$-bimodule $E$ has a \emph{biflag} (with respect to $x, xs$) if
it admits filtrations
\begin{equation} \label{eq:biflag}
C \into E \onto D\quad \text{and} \quad D' \into E \onto C'
\end{equation}
such that $C, C'$ (resp. $D, D'$) are isomorphic to direct sums of
shifts of $R(x)$ (resp. $R(xs)$). If $E$ has a
biflag then $C = E^!_x$, $C' = E_x$, $D = E_{xs}$, $D' = E^!_{xs}$.
Hence the filtrations \eqref{eq:biflag} are canonical, if they exist.

\begin{proof}[Proof of Proposition \ref{Bxxs}]
  For any Soergel bimodule $B$, $B_{x,xs}$ has a biflag (see the last
  three lines of the proof
  of \cite[Proposition 6.4]{Soe3}). It follows from the proposition
  below that $B_{x,xs}$ is isomorphic to a direct sum of
  shifts of $R(x)$, $R(xs)$ and $R(x,xs)$. Finally, one can use
  \cite[Lemma 6.10]{Soe3} to rule out any occurrences of $R(xs)$.
\end{proof}

\begin{prop}
  Suppose that $E$ is a graded $R$-bimodule, and that $E$ has a
  biflag. Then $E$ is isomorphic to a direct sum of
  shifts of $R(x)$, $R(xs)$ and $R(x,xs)$.
\end{prop}

We are grateful to Wolfgang Soergel for providing the following proof.

\begin{proof}
To simplify notation, set $v = x$ and $w = xs$. It will be clear in
the proof that the only assumptions  we need on $v, w$ is that
$\dim (\Gr_v \cap \Gr_w) + 1 = \dim \Gr_w = \dim \Gr_v$.
In the proof, $\Ext^1$ refers to degree zero extensions of graded $R\otimes
R$-modules.

Each choice of linear form $\theta \in (\hg \oplus \hg)^*$ with
$\theta_{|\Gr_v} \ne 0$ and $\theta_{|\Gr_w} = 0$ gives an extension
\begin{equation} \label{eq:gen}
R(v)[-2] \stackrel{\theta \cdot}{\into} R(v,w) \onto R(w).
\end{equation}
Moreover, if we let $I$ denote the regular functions on $\Gr_v \cap
\Gr_w$ then, by \cite[Lemma 5.8]{Soe3}, we have an identification of graded $I$-modules
\begin{equation} \label{eq:id}
\bigoplus_{m \in \ZM} \Ext^1(R(w), R(v)[-2+m]) = I
\end{equation}
mapping the class of \eqref{eq:gen} to $1 \in I$. We fix
such a $\theta$ and hence an identification \eqref{eq:id}.

Let us fix sequences $m_1 \ge m_2 \ge \dots \ge m_f$ and $n_1 \ge
\dots \ge n_g$ of integers.
By additivity and \eqref{eq:id}, we have an identification
\begin{align*}
\Ext^1( \bigoplus_{1 \le j \le
g} R(w)[n_j], &\bigoplus_{1 \le i \le f} R(v)[m_i-2]) = \\ & = \bigoplus_{i,
j} \Ext^1(R(w)[n_j], R(v)[m_i-2]) = \bigoplus_{i,j} I^{m_i-n_j}.
\end{align*}

So now assume that $E$ has a biflag. In particular, there exists a
(homogenous degree zero) extension
\begin{equation} \label{eq:ext}
\bigoplus_{1 \le i \le f} R(v)[m_i-2] \into E \onto \bigoplus_{1 \le j \le
g} R(w)[n_j]
\end{equation}
for certain $m_i, n_j$ as above. Via the above identification such an
extension is determined by a matrix with entries
\[
C_{ji} \in \Ext^1(R(w)[n_j], R(v)[m_i-2]) = I^{m_i - n_j}.
\]
Because $I^{<0} = 0$ and $I^0 = \RM$ it follows that 
our matrix is block upper-triangular (i.e. $C_{ji} =0$ if $m_i < n_j$) with scalar
matrices on the diagonal (i.e. $C_{ji} \in I^0 = \RM$ if $m_i =
n_j$).

Now $I$ is even and so $I^{m_i-n_j} = 0$ if
$m_i$ and $n_j$ are not of the same parity. In particular we may
assume without loss of generality that all $m_i, n_j$ are of the same
parity. Moreover, if there exists $i, j$ with $n_j = m_i$ and $0 \ne
C_{ji} \in \RM$ then we change bases on the left and right of
\eqref{eq:ext} above to ensure that $C_{ji} = 1$ and $C_{ji'} = 0 =
C_{j'i}$ for all $i' \ne i$, $j' \ne j$. In this case our extension
decomposes as $E = R(v,w) \oplus E'$ and we can continue with $E'$ in
place of $E$.

Hence we may assume without loss of generality that our matrix is
block upper-triangular (i.e. $C_{ji} =0$ if $m_i < n_j$) with zeroes on the
diagonal (i.e. $C_{ji} = 0$ if $m_i = n_j$). Under these new assumptions we see that if $m_1
\le n_1$ then $C_{1i} = 0$ for all $i$ and so our
  extension splits as $E = R(w)[n_1] \oplus E'$; again we are done
  by induction.
So we may assume that $m_1 > n_1$. By assumption $E$ has
the biflag property, and hence if we consider the filtration
\[
E^!_w \into E \onto E/E^!_w
\]
we can find isomorphisms $E^!_w \cong \bigoplus R(w)[n'_j - 2]$
with $n'_1 \ge \dots \ge n'_g$ and $E/E^!_w \cong
\bigoplus R(v)[m'_i]$ with $m'_1 \ge \dots \ge
m'_f$.

Multiplication by $\theta$ and the canonical quotient map give injections
\[
E/E^!_w[-2] \stackrel{\theta\cdot}{\into} E^!_v \into E/E^!_w.
\]
(The first map is injective because $E/E_w^!$ is isomorphic to a
direct sum of shifts of $R(v)$, upon which multiplication by $\theta$ is
injective. For the second map, note that every
non-zero element of $E$ is either non-zero in $E_w$ or is contained in
$E_v^!$, and thus has support containing either $\Gr_v$ or
$\Gr_w$. Hence $E_v^! \cap E_w^! = 0$, which implies that the second map is injective.)
Similarly, after choosing $\kappa \in (\hg \oplus \hg)^*$ with
$\kappa_{|\Gr_w} \ne 0$ and $\kappa_{|\Gr_v} = 0$  we have
injections
\[
E/E^!_v[-2] \stackrel{\kappa\cdot}{\into} E^!_w \into E/E^!_v.
\]
By Lemma \ref{lem:inc} below we have $m_i' \in \{ m_i, m_i-2\}$
and $n_j' \in \{ n_j, n_j + 2 \}$. If we consider the graded rank
of $E$ as an $R$-module we deduce
\[
\sum v^{-m_i+2} + \sum v^{-n_j} = \sum v^{-m'_i} + \sum v^{-n'_j + 2}.
\]
Under our assumption $m_1 > n_1$, we see (by considering terms of
minimal degrees on both sides) that $m_1' = m_1$ is impossible and
hence $m_1' = m_1-2$. Hence the smallest non-zero degree of $E/E^!_w$ is
$2 - m_1$ and the injection
\[
R(v)[m_1-2] \into \bigoplus R(v)[m_i-2]  = E^!_v \into E/E^!_w
\]
splits. The result now follows by induction on the graded rank of $E$.
  \end{proof}

\begin{lem} \label{lem:inc}
  Suppose that $m_1 \ge m_2 \ge\dots \ge m_f$ and $m_1' \ge m_2' \ge\dots
  \ge m_g'$ and that we have an injection
\[
\bigoplus R[m_i] \into \bigoplus R[m_i'].
\]
Then $f \le g$ and $m_i \le m_i'$ for all $1 \le i \le f$.
\end{lem}

\begin{proof}
Left to the reader.
\end{proof}

\subsection{Statements of local Hodge theory}

The proof of local hard Lefschetz is an induction relying on some
auxiliary statements which are interesting in their own right. In
this section we state these properties.

Let $(B, \langle -, - \rangle)$ denote a polarised
  Soergel bimodule. We say that $B$ \emph{satisfies local hard
    Lefschetz} (resp. \emph{satisfies local HR}) if for all $x \in W$
  the pair $(A \otimes_R B_x, A \otimes_R \langle -, - \rangle_B^x)$ satisfies hard Lefschetz
  (resp. satisfies hard Lefschetz and HR). We
say that $B$ \emph{satisfies local HR with standard signs} if $B$ satisfies local HR
and for all $x \in W$ we have
\[
(-1)^{\ell(x)} \sigma(\langle c, c \rangle_B^x) > 0
\]
where $0 \ne c \in B_x$ denotes an element of minimal degree.
(The term on the left hand side is a scalar times a power of
$z$; our notation means that this scalar is positive).

To simplify notation in inductive steps we employ the following
notation:
\begin{gather*} 
hL(y): \begin{array}{c} \text{$B(y)$ satisfies local hard Lefschetz.} \end{array} \\
HR(y): \begin{array}{c} \text{$B(y)$ satisfies local HR with standard
    signs.} \end{array}\end{gather*}
(As always we regard $B(y)$ as polarised with respect to its
intersection form.)
Given a subset $X \subset W$ we write $hL(X)$ (or $hL(\le x)$
etc.) to mean $hL(x)$ for all $x \in X$ etc.

Fix $s \in S$. 
In \S\ref{sec:SP1} we explained how to associate to a polarised
Soergel bimodule $(B, \langle -, - \rangle)$ and $x \in W$
with $x < xs$ a polarised $\PM^1$-sheaf $M(B,x,xs)$. We say that $B$ \emph{satisfies local hard Lefschetz}
  (resp. \emph{satisfies local HR}) \emph{in the $s$ direction} if:
  \begin{enumerate}
  \item $B$ satisfies local hard Lefschetz (resp. local HR);
  \item for all $x \in W$ with $x < xs$ the polarised $\PM^1$-sheaf $M(B,x,xs)$
  satisfies hard Lefschetz (resp. satisfies HR).
  \end{enumerate}

We abbreviate:
\begin{gather*}
  hL(y)_{s}: \begin{array}{c} \text{$B(y)$ satisfies local hard
      Lefschetz in the $s$ direction.} \end{array} \\
HR(y) _{s}: \begin{array}{c} 
\text{$HR(y)$ holds and $B(y)$ satisfies local HR in the $s$ direction.} \end{array}
\end{gather*}

\section{Proof} \label{proof}

\subsection{Outline of the proof}

With the terminology of the previous section the main result of this paper is:

\begin{thm}  \label{thm:main}
For all $y \in W$, $HR(y)$ holds.
\end{thm}

We now outline the structure of the argument. Throughout, $y \in W$
and $s \in S$ is a simple reflection.

The following are the key statements, which rely on weak Lefschetz
style induction:

\begin{claim}[Proposition \ref{wLnd}] \label{c1}
$HR(<y) \Rightarrow
  hL(y)$. \end{claim}

\begin{claim}[Proposition \ref{wLd}] \label{c2} If $ys > y$, $HR(<y)_{s} + HR(y) \Rightarrow hL(y)_{s}$.
\end{claim}

The following are ``limit lemma'' style arguments, which are easier:

\begin{claim} \label{c3} If $ys > y$ then $HR(y) + hL(y)_s + hL(\le ys) \Rightarrow
  HR(ys)$. \end{claim}

\begin{proof} Firstly, if $B(ys)_x$ satisfies HR, then it
  satisfies HR with standard signs, by Theorem \ref{thm:em} and Corollary
\ref{cor:HRamb}. Hence it is enough to check that $B(ys)_x$ and $B(ys)_{xs}$ satisfy
HR, for all $x < xs$. Because $ys > y$, $B(y)B(s)$ is perverse (see Remark \ref{rem:bbs}), $B(ys)$ is a summand 
of $B(y)B(s)$ and we have an
isometry (see Lemma \ref{lem:isom})
\begin{equation} \label{eq:stalk}
A \otimes_R (B(y)B(s))_x = M_{0,\infty}
\end{equation}
where $M = M(B(y),x,xs)$ is the $\PM^1$-sheaf associated to $B(y)$,
and $x < xs$. Moreover, by Proposition \ref{prop:indint} we have a canonical
  identification $B(y)B(s)_x = B(y)B(s)_{xs}$ which is $-1$ times an
  isometry (i.e. $\langle -, - \rangle^x = -\langle -, -\rangle^{xs}$
under this identification). By Lemma \ref{lem:orthog}, $B(ys)$ is an orthogonal summand of
$B(y)B(s)$ and we can apply Lemma
\ref{lem:HRsummand} to conclude that it is enough to prove that
$M_{0,\infty}$ satisfies $HR$ for all $x$ as above.

In other words, if multiplication by $(z,z)$ on
$M_{0,\infty}$ satisfies HR, then $HR(ys)$ holds. By our
assumptions $hL(y)_s$ and $hL(\le ys)$, multiplication by $(az,z)$ on $M$
satisfies hard Lefschetz for all $a \ge 1$. By
$HR(y)$, the polarised $\PM^1$-sheaf $M$ is easily seen to have opposite signs
($M_{0,\infty}$ is
generated in degrees $\le 0$ by \eqref{eq:stalk} and the fact that
$B(y)B(s)$ is perverse), and now the
result follows by Corollary \ref{cor:lim}.
\end{proof}

\begin{claim} \label{c5}
If $ys > y$ then $hL(y)_s + HR(y) \Rightarrow HR(y)_s$.
\end{claim}

\begin{proof}
  Let $x < xs$ and $M = M(B(y),x,xs)$ be as in the previous
  proof.  We need to check that $M$ satisfies HR. We saw in the
  previous proof that $M_{0,\infty}$ is generated in degrees $\le
  0$. Now:
  \begin{enumerate}
  \item $HR(y)$ implies that $M$ is polarised with opposite signs;
  \item $hL(y)_s$ implies that multiplication by $(az,z)$ on
    $M_{0,\infty}$ satisfies hard Lefschetz for $a > 1$.
  \end{enumerate}
The result now follows from Corollary \ref{cor:lim} (with $I = (1,\infty)$).
\end{proof}

The following is straightforward (``constant case''):

\begin{claim}[Proposition \ref{prop:constant}] \label{c4} If $ys < y,
  HR(\le y) \Rightarrow HR(y)_{s}$.
\end{claim}

From these claims we deduce:

\begin{claim} \label{step1}
  If $ys > y$ then $HR(<ys) + hL(y)_s \Rightarrow HR(ys)$.
\end{claim}

\begin{proof}
\begin{gather*}
HR(<ys) + hL(y)_s  \stackrel{\text{Claim \ref{c1}}}{\Longrightarrow}
HR(<ys) + hL(\le ys) + hL(y)_s\\
 \Rightarrow HR(y) + hL(\le ys) + hL(y)_s 
\stackrel{\text{Claim \ref{c3}}}{\Rightarrow}
HR(ys). \qedhere
\end{gather*}
\end{proof}

We also have:

\begin{claim} \label{step2}
  $HR(y) + HR(<y)_s \Rightarrow HR(y)_s$.
\end{claim}

\begin{proof}
  If $ys < y$ then this follows from Claim \ref{c4} (remember that
  $HR(<y)_s$ includes $HR(<y)$ by definition). So we can assume
  $ys > y$. Now $hL(y)_s$ holds by Claim \ref{c2} and then we are done by
  Claim \ref{c5}.
\end{proof}

Now we can give the proof of Theorem \ref{thm:main} (assuming the
above statements):

\begin{proof}[Proof of Theorem \ref{thm:main}]
Let $X$ denote  an ideal in the Bruhat order and for all $x \in X$ assume $HR(x)$
and $HR(x)_s$ for all $s \in S$. If $X \ne W$ then we can
  choose $y' \in W \setminus X$ of minimal length and $s \in S$ with $y := y's < y'$.
Now $y \in X$ and Claim \ref{step1}
tells us that $HR(ys)$ holds and then Claim \ref{step2} tells us that
$HR(ys)_t$ holds for all $t \in S$. Hence we can add $ys$ to our set $X$.

One may check directly that $HR(\id)$ and $HR(\id)_s$ hold for all $s
\in S$. The above induction tells us that $HR(x)$ and $HR(x)_t$ hold for
all $x \in W$ and $t \in S$. The theorem now follows.
\end{proof}

\subsection{Easy cases} In this section we make some easy
observations which are used in the proof.

\begin{lem} \label{lem:pp}
  Suppose that $B$ is positively polarised and that $HR(y)$ holds for
  all indecomposable summands $B(y)$ of $B$. Then $B$ satisfies local HR.
\end{lem}

\begin{proof}
Consider the canonical decomposition $B = \bigoplus V(y)
  \otimes_{\RM} B(y)$ of \S \ref{sec:pp}. Let
  $z$ be maximal such that $V(z) \ne 0$. Fix $x \in W$. We want to show
  that $(A \otimes_R B_x, A \otimes_R \langle -, - \rangle_B^x)$
  satisfies HR. Our  decomposition induces a
  decomposition $A \otimes_R B_x = \bigoplus A \otimes_R
  (V(y) \otimes_{\RM} B(y))_x$. Now let $p \in A \otimes_R
  (V(y) \otimes_{\RM} B(y))_x$ be a
  primitive element in degree $d = -\ell(y) + 2d'$.  Let $c = z^{-d}\langle p, p \rangle^x
  \in \RM$. Then, by $HR(y)$ and the
  definition of positively polarised (see \S \ref{sec:pp}) we see that
\[
0 < (-1)^{(\ell(z)-\ell(y))/2} (-1)^{d' + \ell(x)}  c = (-1)^{(\ell(z)
  + d)/2 + \ell(x)}c.
\]
Hence the sign of $c$ depends only on $\ell(z)$, $\ell(x)$ and $d$, and not on
$y$, and hence $B$ satisfies local HR.\end{proof}

The proof of the following analogue of the previous lemma for
$\PM^1$-sheaves is similar, and is left to the reader.

\begin{lem} \label{lem:ppP1}
  Suppose that $B$ is positively polarised and that $HR(y)_s$ holds for
  all indecomposable summands $B(y)$ of $B$. Then $B$ satisfies local
  HR in the $s$ direction.
\end{lem}

\begin{prop} \label{prop:constant}
Suppose that $ys < y$ with $y \in W$ and $s \in S$. If $HR(\le y)$ holds
then for all $x < xs$ the polarised $\PM^1$-sheaf
$M(B(y),x,xs)$ is polarised constant, and hence satisfies HR.
\end{prop}

\begin{proof} Our first step is to prove that the $\PM^1$-sheaf
    $M := M(B(ys)B(s),x,xs)$ is polarised constant. Let us check that that the
    stalks are generated in degrees $\le -2$. Because of the shift
    involved in the definition of $M$,
    this is equivalent to checking that $B(ys)B(s)_x$ and $B(ys)B(s)_{xs}$ are
    generated in degrees $\le -1$. However $B(ys)B(s)$ is perverse
    (see Remark \ref{rem:bbs}),
    and it follows from \cite[Theorem 5.3]{Soe3} and the solution of Soergel's
    conjecture that the stalks of
    any $B(z)$ with $z \ne \id$ are generated in degrees $\le
    -1$. Thus the claim follows from the fact that $B(\id)$ is not
    a summand of $B(ys)B(s)_x$ (all summands $B(z)$ satisfy $zs < z$).

Note that $B(ys)B(s)$ is positively polarized by
  Remark \ref{rem:bbs}. By Lemma \ref{lem:pp}
  and our assumption $HR(\le y)$, we conclude that $B(ys)B(s)$
  satisfies local HR. By Proposition \ref{prop:BB(s)} we have
\[
 (B(ys)B(s))_x = B(ys)_{x,xs}[1] = (B(ys) B(s))_{xs}.
\]
By Proposition \ref{prop:indint} we see that under the above
identifications we have
\[
\langle -, -\rangle_{B(ys)B(s)}^x = - \langle -, -\rangle_{B(ys)B(s)}^{xs}. 
\]
It follows that the $\PM^1$-sheaf $M$ is polarised constant.

Now we have a canonical and orthogonal
decomposition (see Lemma \ref{lem:orthog})
\[
B(ys)B(s) = B(y) \oplus E
\]
where $E$ is a polarised Soergel bimodule with all indecomposable
summands isomorphic to $B(z)$ for $z < ys <y$. Now it is
  not difficult to see
that any orthogonal summand of a polarized constant sheaf is
polarized constant. In particular the summand $M(B(y),x,xs)$ of $M$ is polarised
constant. Hence $M(B(y),x,xs)$ satisfies HR by Lemma
\ref{lem:const}.
\end{proof}

\begin{prop} \label{prop:yy}
For $y \in W$, $(A \otimes_R B(y)_y, A \otimes_R \langle -, -
  \rangle_{B(y)}^y)$ satisfies HR with standard signs.\end{prop}

\begin{proof}
  By Soergel's character formula \cite[Theorem 5.3]{Soe3},
$B(y)_y$ is free of graded rank
  $v^{-\ell(y)}$. Thus $(A \otimes_R B(y)_y, A \otimes_R \langle -, -
  \rangle_{B(y)}^y)$ satisfies hard Lefschetz, as this is automatic
  for $A$-modules of rank 1. Moreover, in the notation of \S\ref{sec:em},
  $c_{y,y} \in B(y)_y$ is a generator and
\[
\langle c_{y,y}, c_{y,y} \rangle_{B(y)}^y = \gamma e_{y,y} = \gamma
(-1)^{\ell(y)} \prod_{ t \in L_T(y)} \frac{1}{\a_t} \qquad \text{for
  some $\g \in \RM_{>0}$.}
\]
by Theorem \ref{thm:em} and Proposition \ref{prop:equi}(2). Applying
$\sigma$ (and using that $\sigma(\alpha) > 0$ for $\alpha \in \Phi^+$)
yields the result.
\end{proof}

\begin{remark}
  More generally, the above proof works whenever $B(y)_x$ is free of
  rank 1 (the ``rationally smooth case'').
\end{remark}

\subsection{Non-deformed case}

\begin{prop} \label{prop:factor1}
  Suppose that $\un{y} = s_1 \dots s_m$ is a reduced expression and $\l \in \hg^*$ is such
  that
  \[
  \langle s_{i+1} \dots s_m(\l), \alpha_{s_i}^\vee \rangle > 0
  \]
  for all $1 \le i \le m$. Then there exists a positively polarised Soergel bimodule $(B', \langle -,
  - \rangle_{B'})$ all of whose summands are isomorphic to $B(x)$ with
  $x < y$ and a map
  \[
  d : B(y) \to B'[1]
  \]
  such that
  \begin{equation}
    \label{adjoint rel}
d^* \circ d = B(y) \l- (y\l)  B(y)     
  \end{equation}
(as always, $B(y)$ is polarised with its intersection form.)
\end{prop}

The proof follows the same lines as the proof of \cite[Theorem 6.21]{EW}. During
the proof we need the perverse filtration and the functors $\tau_{\le
  i}$ of \cite[\S 6.3]{EW}.

\begin{proof} We prove the proposition by induction on $m$. The
  statement makes sense for $m = 0$ (so $y = \id$). In this case we can take $B' = 0$.

Now assume $m \ge 1$. Let $z = ys_m$ and $s = s_m$ so $y = zs$. Then
we can apply induction with $\un{z} = s_1 \dots
  s_{m-1}$ and $s\l \in \hg^*$ to find a positively polarised bimodule
  $(D, \langle -, - \rangle_{D})$ and a map
\[
d' : B(z) \to D[1]
\]
such that
\[
(d')^* \circ d' = B(z) s(\lambda) - z(s(\lambda)) B(z)
= B(z) s(\lambda) - y(\lambda) B(z) .\]

Now consider the map
\[
d = \left ( \begin{matrix} \sqrt{\langle \l, \alpha_{s}^\vee
      \rangle} B(z)m \\
    d'B(s) \end{matrix} \right ) : B(z)B(s) \to (B(z) \oplus DB(s))[1].
\]
(The target is polarised with respect to the intersection form on
$B(z)$ and the induced form on $DB(s)$.) We have
$m^* = \mu$ (see \eqref{eq:duals}) and hence the adjoint of $d$ is (see
\eqref{eq:prodadjoint})
\[
d^* = \left ( \sqrt{\langle \l, \alpha_{s}^\vee
      \rangle} B(z) \mu \quad (d')^* B(s) \right ) \]
and hence
\begin{align*}
d^* \circ d &= \langle \l, \alpha_{s}^\vee
      \rangle B(z) (\mu \circ m) + ((d')^* \circ d')B(s) \\
& = \langle \l, \alpha_{s}^\vee
      \rangle B(z) (\mu \circ m) +  B(z)
      s(\lambda) B(s) - y(\lambda) B(z)B(s) \\
& = B(z)B(s) \l - y(\lambda) B(z)B(s)
\end{align*}
by the ``polynomial sliding'' relation \eqref{slide}.

In particular $d$ satisfies the relation \eqref{adjoint rel}. We now
need to show that we can replace $B(z) \oplus DB(s)$ by a perverse
summand whilst keeping the relation \eqref{adjoint rel}.

Let us choose a decomposition $D = \bigoplus B(u)^{\oplus m_u}$ with $m_u
\in \ZM_{\ge 0}$ and define
\[
D^{\uparrow} = \bigoplus_{us > u} B(u)^{\oplus m_u} \quad\text{and}
\quad  D^{\downarrow} = \bigoplus_{us < u} B(u)^{\oplus m_u}.
\]
Then we have a canonical, orthogonal decomposition (see Lemma \ref{lem:orthog})
\[
D = D^{\uparrow} \oplus D^{\downarrow}.
\]
The bimodule $D^\uparrow B(s)$ is perverse. We have (canonically and orthogonally)
\begin{equation} \label{eq:decomp}
B(z)B(s) = B(y) \oplus E
\end{equation}
for some perverse Soergel bimodule $E$ (see Remark
\ref{rem:bbs}). Moreover the restriction of the
  intersection form on $B(z)B(s)$ yields the intersection form on
  $B(y)$, up to a positive scalar multiple. By rescaling the inclusion
  $B(y) \into B(z)B(s)$ if necessary we can assume that this
  scalar multiple is 1.

We have a (non-canonical and non-orthogonal) decomposition
\[
D^{\downarrow} B(s) = D^{\downarrow}[1] \oplus D^{\downarrow}[-1].
\]
Consider the maps induced by $d$ and $d^*$ on the (canonical) summands $B(y)
\subset B(z)B(s)$ and $D^{\downarrow}B(s) \subset B(z) \oplus DB(s)$:
\[
B(y) \stackrel{f}\to D^{\downarrow} B(s) [1] \stackrel{f^*}\to B(y)[2].
\]
All summands of $D^{\downarrow}$ are isomorphic to $B_x$ with
$x < y$. Hence $f$ lands in
\[
\tau_{\le -1} (D^{\downarrow} B(s) [1]) = D^{\downarrow}[2].
\]
Similarly, $f^*$ is zero on $\tau_{\le -1} (D^{\downarrow}
B(s) [1])$. In particular:
\begin{equation} \label{eq:comp}
f^* \circ f = 0.
\end{equation}

Let us write the matrix of $d$ with respect to these
decompositions as
\[
d = \left ( \begin{matrix} a & b \\ c & d \\ f & g \end{matrix} \right
) : B(y) \oplus E \to (B(z)\oplus D^\uparrow B(s) \oplus
D^{\downarrow} B(s))[1].
\]
Then
\[
d^* = \left (\begin{matrix} a^* & c ^* & f^* \\ b^* & d^* &
    g^* \end{matrix} \right).
\]
The computation of $d^* \circ d$ above and \eqref{eq:comp} imply that
\begin{equation}
  \label{eq:ac}
  a^* \circ a + c^* \circ c  = 
  a^* \circ a + c^* \circ c + f^* \circ f = B(y) \lambda
  - y(\lambda) B(y).
\end{equation}

Now define $d_{sub}$ to be the composition
\[
B(y) \to B(z)B(s) = B(y) \oplus E  \stackrel{d}{\to} (B(z)\oplus D^\uparrow B(s) \oplus
D^{\downarrow} B(s))[1] \to (B(z)\oplus D^\uparrow B(s))[1]
\]
where the first (resp. last map) is the inclusion (resp. projection)
with respect to the above decompositions. By the
orthogonality of these decompositions the adjoints of the first (resp. last) map is the
projection (resp. inclusion). By \cite[Proposition 6.12]{EW} the bimodule
$D^{\uparrow}B(s)$ is positively polarised, and $B(z)$ is clearly
positively polarised.

Finally, by \eqref{eq:ac}:
\[
d_{sub}^* \circ d_{sub} = a^*\circ a + c^*\circ c = B(y) \lambda - y(\lambda) B(y).
\]
Now we are done: we can take $d = d_{sub}$ and $B' = B(z) \oplus D^\uparrow B(s)$ which
is perverse and positively polarised. (It is easy to check that the
signs on the two summands match up.)
\end{proof}

\begin{prop} \label{wLnd}
  $HR(<y) \Rightarrow
  hL(y)$. 
\end{prop}

\begin{proof} The fact that hard Lefschetz is true for $A \otimes_{R} B(y)_y$ follows from
  Proposition \ref{prop:yy}.

It remains to show hard Lefschetz for $x
  < y$. Let us apply the above proposition with $\l =
  \rho$. Taking the stalk at $x$ and applying $A \otimes_R (- )$ we see that
  we have a map
\[
A \otimes_R B(y)_x \stackrel{d}{\to} A \otimes_R B'_x[1]
\]
of free $A$-modules equipped with forms $A \otimes
\langle -, - \rangle_{B(y)}^x$ and $A \otimes \langle -, -
\rangle_{B'}^x$ which are symmetric and non-degenerate over $K$ and such that
$d^* \circ d$ is equal to left multiplication by $ \sigma( x(\l) -
  y(\l))$.

By Lemma \ref{lem:rho}, $ \sigma( x(\l) -
  y(\l)) > 0$ and in particular is non-zero. 
Moreover, as $B'$ is positively polarised, Lemma \ref{lem:pp}
ensures that $A \otimes_R B_x'$ satisfies HR. Now we can apply
Proposition \ref{prop:wL} to conclude that $A \otimes_R B(y)_x$
satisfies hard Lefschetz. The proposition follows.
\end{proof}

\subsection{Deformed case}

\begin{prop} \label{prop:def}
  Suppose that $\un{y} = s_1 \dots s_m$ is a reduced expression and $s \in S$ is such that
  $ys > y$. Let $\l \in \hg^*$ be such that
\[
\langle \l, \alpha_s^\vee \rangle > 0
\]
and
  \[
  \langle s_{i+1} \dots s_ms(\l), \alpha_{s_i}^\vee \rangle > 0
  \]
  for all $1 \le i \le m$.

Then for any $0 \le a < 1$ 
 there exists a positively polarised bimodule $(B', \langle -,
  - \rangle_{B'})$ all of whose summands are isomorphic to
    $B(x)$ with $x <
    y$ and a map
  \[
  d : B(y)B(s) \to (B(y) \oplus B'B(s))[1]
  \]
  such that
  \begin{equation}
    \label{adjoint rel2}
d^* \circ d = 
B(y)B(s) \l - B(y) (a s(\l)) B(s) - (1-a)ys(\l) B(y)B(s)
  \end{equation}
($B(y)$
  is polarised with its intersection form).
\end{prop}

\begin{proof}
  We want to find $d$ such that
\begin{gather} \label{eq:want}
d^* \circ d = 
B(y)B(s) \l - B(y) (a s(\l)) B(s) - (1-a)ys(\l) B(y)B(s)
\nonumber \\ = B(y) (\mu \circ m) \langle \l, \alpha_s^\vee \rangle +
B(y)(1-a)s(\l)B(s)- (1-a) ys(\lambda) B(y) B(s)
\end{gather}
(we have used \eqref{slide}). 
Applying Proposition \ref{prop:factor1} with $y = s_1 \dots s_m$ and
$(1-a)s(\lambda) \in \hg^*$ gives us a positively polarised 
bimodule $B'$, all of whose summands are isomorphic to
    $B(x)$ with $x <
    y$, and a map
\[
d' : B(y) \to B'[1]
\]
such that
\[
(d')^* \circ d' = 
B(y)(1-a)s(\l) - (1-a) ys(\lambda) B(y).
\]
Now if we set
\[
d := \left ( \begin{matrix} \sqrt{\langle \l, \alpha_s^\vee \rangle}
    B(y) m \\
    d' B(s) \end{matrix} \right )
: B(y)B(s) \to (B(y) \oplus B'B(s))[1]
\]
then
\[
d^* = \left ( \begin{matrix} \sqrt{\langle \l, \alpha_s^\vee
      \rangle} B(y) \mu & (d')^* B(s) \end{matrix} \right )
\]
and
\begin{align*}
d^* \circ d & = \langle \l, \alpha_s^\vee \rangle B(y) (\mu
\circ m) + ((d')^* \circ d') B(s)  \\
& = \langle \l, \alpha_s^\vee \rangle B(y) (\mu
\circ m) + B(y)(1-a)(s\l) B(s) - (1-a)(ys\l)B(y)B(s)
\end{align*}
as required.
\end{proof}

\begin{prop} \label{wLd}
If $ys > y$, $HR(<y)_{s}  + HR(y) \Rightarrow hL(y)_{s}$.
\end{prop}

\begin{proof} Let $\l \in \hg^*$ and
\[
d : B(y) B(s) \to ( B(y) \oplus B'B(s))[1]
\]
be as in the statement of the previous proposition (for some fixed $0
\le a < 1$).
  Fix $x < xs$ and let us take the stalk of $x$ (we abuse
  notation and continue to denote these maps by the same symbols):
\begin{align*}
d : (B(y)B(s))_x &\to (B(y)_x\oplus B'B(s)_x)[1] \\
d^* : B(y)_x\oplus B'B(s)_x &\to (B(y)B(s))_x[1].
\end{align*}

We claim that we can obtain $A \otimes_R d$ and $A \otimes_R d^*$ as the global sections
of a pair of adjoint maps
\begin{equation} \label{eq:maps}
\tilde{d} : M \to N[1]  \quad \text{and} \quad \tilde{d}^* : N \to M[1]
\end{equation}
of polarised $\PM^1$-sheaves. Let  $B'$ and $d'$ be as in the proof of
the previous proposition.  Consider the following polarised $\PM^1$-sheaves:
\begin{enumerate}
\item $M := M(B(y),x,xs)$, polarised as in \S \ref{sec:SP1}.
\item $N'$ := the skyscraper at 0 with stalk $A \otimes_R B(y)_x$
  (i.e. $N'_0 = A \otimes_R B(y)_x$, $N_{\CM^*} = 0$, $N'_\infty = 0$ and
  polarisation $\sigma( \langle -, -\rangle^x_{B(y)})$ on $N'_0$). 
\item $N'' := M(B',x,xs)$, polarised as in \S \ref{sec:SP1}.
\end{enumerate}

We have a natural map $\tilde{d}_1 : M \to N'[1]$ given by $(b_0,
b_\infty) \mapsto \sqrt{\langle \l, \alpha_s^\vee \rangle}(b_0, 0)$. The adjoint $\tilde{d}_1^*$ is given by
$(b_0, 0) \mapsto \sqrt{\langle \l, \alpha_s^\vee
  \rangle}(\sigma(x(\a_s))b_0, 0)$. Under the identification of
\eqref{eq:globsec} one checks that on global sections the maps $\tilde{d}_1$ and
$\tilde{d}_1^*$ agree with the maps $A \otimes_R (B(y)B(s))_x \to A \otimes_R B(y)_x[1]$ and
$A \otimes_R B(y)_x \to A \otimes_R  (B(y)B(s))_x[1]$ induced by $\sqrt{\langle \l,
  \alpha_s^\vee \rangle}B(y)m$ and $\sqrt{\langle \l, \alpha_s^\vee
  \rangle}B(y)\mu$.

The map $d' : B(y) \to B'[1]$ induces a map $\tilde{d}_2: M \to N''[1]$
of polarised $\PM^1$-sheaves. We denote its adjoint by $\tilde{d}_2^*$. Under
the identification \eqref{eq:globsec},  the map $\tilde{d}_2$ agrees on
global sections with the map $A \otimes_R (B(y)B(s))_x \to A \otimes_R (B'B(s))_x[1]$
induced by $d'B(s)$. By Lemma \ref{lem:isom}, $\tilde{d^*_2}$ agrees
on global sections with the map $A \otimes_R (B'B(s))_x \to A
\otimes_R (B(y)B(s))_x[1]$ induced by $(d')^*B(s)$.
Hence if we set $N := N' \oplus N''$ and $\tilde{d} := \tilde{d_1} +
\tilde{d_2}$ we have constructed our desired maps in \eqref{eq:maps}.

By the previous proposition 
\[
d^* \circ d = 
B(y)B(s) \l - B(y) (a s(\l)) B(s) - (1-a)ys(\l) B(y)B(s)
\] 
and hence, under the injection $B(y)B(s)_x
\into B(y)_x \oplus B(y)_{xs}$, Lemma \ref{lem:indmult} implies that
$d^* \circ d$ agrees with multiplication by
\[
(x(\l) - a (xs(\l)) - (1-a)(ys(\l)), xs(\l) - a (xs(\l)) - (1-a)(ys(\l))).
\]
Hence for all $b \in M_{0,\infty}$ we have the relation
  \[
  (d^* \circ d)(b) = \g \cdot b
  \]
  where $\g = (\l_0, \l_\infty) \in Z^2$ is given by
  \begin{gather} \label{eq:l0li}
    \l_0 = \sigma((x - ys)(\l) - a(xs - ys)(\l)), \\
    \l_\infty = \sigma((xs-ys)(\l) - a(xs - ys)(\l)) =
    \sigma((1-a)(xs-ys)(\l)). \nonumber
  \end{gather}
Using Lemma \ref{lem:rho}, one may check that
  $\g \in Z^2_\amp$ for all dominant regular $\l$ and $0 \le a <1$.

Now $B'$ is positively polarised and  by our assumption
$HR(<y)_s$,  $B'$
satisfies HR in the $s$ direction,  by Lemma \ref{lem:ppP1}. Hence $N''$ satisfies HR. Also
$B(y)$ satisfies local HR by assumption, and so $N'$
satisfies HR.
Hence $N$ satisfies HR (one checks easily that the signs
  on $N'$ and $N''$ match). We deduce from Proposition
\ref{prop:wLP1} that $M(B(y),x,xs)$ satisfies hard Lefschetz for all
pairs $(\l_0, \l_\infty)$ above.

We will see in the lemma below that we can vary $\l$ and $a$ so that
$\l_0/\l_\infty$ takes on 
all values in $(1,\infty)$. Hence $hL(y)_s$ holds, by Lemma
\ref{lem:saturation}.
\end{proof}



\begin{lem} For varying dominant regular $\l \in \hg^*$ and $0 \le a < 1$,
  $\l_0/\l_\infty$ (see \eqref{eq:l0li}) takes on all values in $(1,\infty)$.
\end{lem}

\begin{proof}By continuity and the intermediate value theorem it is
  enough to show that by varying $\l$ and $a$ we can get values which
  are both arbitrarily large and arbitrarily close to 1.
  We have
  \[
  \l_0 - \l_\infty = \sigma(x(\l) - xs(\l)) = \langle \l,
  \alpha_s^\vee \rangle \sigma(x(\alpha_s)).
  \]
  and hence
  \[
  \frac{\l_0}{\l_\infty} =   \frac{\l_\infty + (\l_0-\l_\infty)}{\l_\infty} = \frac{(1-a)C + \langle \l, \alpha_s^\vee
    \rangle \sigma(x(\alpha_s))}{(1-a)C}
  \]
  where $C = \sigma((xs-ys)(\l))$. Hence if we choose $\l = \rho$ then
  $\l_0/\l_\infty \to \infty$ as $a \to 1$.

  On the other hand, if we take $a = 0$ and let $\l$ approach the
  $s$-wall (so that $\langle \l, \alpha_s^\vee \rangle \to 0$) then we
  see that $\l_0 / \l_\infty \to 1$. The result now follows.
\end{proof}

\subsection{Soergel's conjecture} \label{sec:sc} In this section we discuss how these
arguments can be adapted to deduce Soergel's conjecture. Unfortunately the most
difficult parts of the proof take the same road as \cite{EW}, so this
cannot be considered a new proof. In fact, in the author's opinion the current paper is
strictly more complicated than \cite{EW}. For this reason we only give a sketch.

In the following (as in \cite{EW}) we fix $x \in W$ and $s \in S$ with
$xs > x$ and assume Soergel's conjecture for all $y < xs$.

\begin{prop}
  If $B := B(x)B(s)$ satisfies local hard Lefschetz then Soergel's
  conjecture holds for $B(xs)$.
\end{prop}

\begin{remark}
This proposition seems to have first been observed by Soergel and
Fiebig a number of years ago. We will see in the proof that the
proposition is not if and only if. It is not clear to the
author how much stronger local hard Lefschetz is. (There is also the
question of the choice of specialisation parameter $\rho^\vee$.)
\end{remark}

\begin{proof} Fix $y < xs$ and consider the inclusion
\[
i_y : B_y^! \into B_y.
\]
Then $B_y^!$ is generated in degrees $\ge \ell(y)$. 
Soergel shows that the image is contained in $B_y \cdot p_y$ for some
(explicit) product of $\ell(y)$ roots $p_y$. Moreover, by \cite[Lemma 7.1(3)]{Soe3},
Soergel's conjecture for $B(xs)$ (under the assumption of Soergel's
conjecture for all $B(y)$ with $y < xs$) is 
equivalent to the above inclusion inducing an inclusion (then
necessarily an isomorphism in degree $\ell(y)$)
\[
(B_y^!)^{\ell(y)} \into \RM \otimes_R (B_y  \cdot p_y)
\]
for all $y < xs$.

If we specialise via $\sigma : R \to A = \RM[z]$ we see that Soergel's
conjecture is equivalent to the natural map inducing an inclusion
\[
(A \otimes_R B_y^!)^{\ell(y)} \into \RM \otimes_R (z^{\ell(y)}
\otimes_A  B_y).
\]
This is the case if and only if
\[
(1 \otimes B_y^!) \cap (z^{\ell(y)+1} \otimes B_y^{-\ell(y)-2}) = 0.
\]
In other words if we set $H := (A \otimes_R B_y) / (z \otimes_R
B_y^!)$ we want
\[
(\ker (z \cdot) \cap H^{\ell(y)} ) \cap (z^{\ell(y)+1} H^{-\ell(y)-2}) = 0
\]
or in other words that multiplication by $z^{\ell(y) + 2}$ should give an isomorphism
$H^{-\ell(y)-2} \to H^{\ell(y) + 2}$. This is clearly the case if $z$
satisfies hard Lefschetz on $H$, which is the case if $B$ satisfies
local hard Lefschetz (essentially by definition, see Lemma \ref{lem:zhl}).
\end{proof}

It seems likely that one could adapt the proof over the last few pages
to prove local hard Lefschetz for $B(x)B(s)$ assuming only
  statements (Soergel's conjecture, local hard Lefschetz and HR,
  \dots) for elements $y \le x$.
One could then use the above proposition to deduce
Soergel's conjecture and continue the induction. However the key
ideas would still be those of \cite{EW} and
this paper is already complicated
enough!

\section{Some calculations in $\mathfrak{sl}_4$} \label{ap}

The goal of this section is to give a few examples
of local intersection forms and see the connection to the Jantzen
filtration.

Let $\mathfrak{g} =  \mathfrak{sl}_4(\RM)$, $\bg \subset \mathfrak{g}$ the Borel subalgebra of
upper-triangular matrices, $\hg \subset \bg$ the Cartan subalgebra of
diagonal matrices. Denote by $\alpha_s, \alpha_t, \alpha_u$ the simple
roots in $\hg^*$ and $s, t, u$ the corresponding simple reflections in
the Weyl group $W$. (Our normalisation is such that $su = us$.) Let
$\alpha_s^\vee, \alpha_t^\vee, \alpha_u^\vee \in \hg$ denote the simple coroots.

Using the realisation $\hg$ we can define the category of Soergel
bimodules for $W$ and the theory of this paper applies. We
work over $\RM$ so that we can discuss signatures.

\subsection{The strategy}

Recall the definition of the local intersection form. We start with a
polarised Soergel bimodule $(B, \langle-,-\rangle_B)$. Then $\langle
-, -\rangle_B$ induces an $R$-valued symmetric
form on the costalk $B_y^!$ by restriction. The inclusion
\[
B_y^! \into B_y
\]
is an isomorphism over $Q$ and realises $B_y^!$ and $B_y$ as dual
lattices in $Q \otimes_R B_y$. The $R$-valued form on $B_y^!$ then
induces a $Q$-valued form on $B_y$, which is the local intersection
form.

We use the following lemma to calculate the local intersection form:

\begin{lem}
  Suppose that $e_1, \dots, e_m$ denotes a graded $R$-basis for
  $B_y^!$, and let $e_1^*, \dots, e_m^*$ denote the dual basis of $B_y$. If $M := ( \langle e_i, e_j \rangle_B)_{1 \le i, j \le m }$
  denotes the Gram matrix of $\langle -, -\rangle_B$ on $B_y^!$ in the
  basis $e_1, \dots, e_m$ then
  the Gram matrix of $\langle -, -\rangle_B$ on $B_y$ in the basis
  $e_1^*, \dots, e_m^*$ is given by $M^{-1}$.
\end{lem}

Hence one needs to calculate a basis of $B_y^!$ and then
compute the restriction of $\langle -, -\rangle_B$ to it. Finding a basis for $B_y^!$ is a linear algebra problem (one knows the
graded rank from a calculation in the Hecke algebra). However this can
be tricky in practice.

Below we will only consider the case $y = \id$, in which case
$B_{\id}^!$ can be calculated easily using Soergel calculus
\cite{EWSC}, as we will see. (Actually the restriction $y = \id$ is not necessary, but
we don't go into that here.) In the following we will use the
notation of \cite[\S 2 and \S 6]{EWSC} concerning expressions and 
light leaves morphisms. We will denote subexpressions by the
corresponding \emph{01-sequence} (see \cite[\S 2.4]{EWSC}).
See \cite[\S 2.10]{HW} for a sample
calculation of local intersection forms
using light leaves morphisms and Soergel calculus.

\subsection{The first singular Schubert variety} We calculate the local intersection form for $B = B(tsut)$ at $y = \id$.

In this case $B(t)B(s)B(u)B(t)$ is
indecomposable (as follows from a calculation in the
Hecke algebra) and hence $B = B(t)B(s)B(u)B(t)$.
There are two subexpressions of $\un{x} = tsut$ for $y = \id$: $\un{e}=
0000$
and $\un{f} = 1001$ of defects 4 and 2 respectively. The corresponding light
leaf maps (with colour coding ${\color{red}s}, t
, {\color{blue}u}$) are as follows:
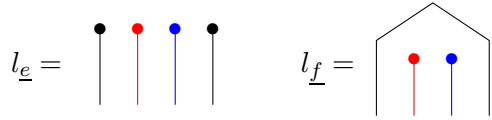
\[
l_{\un{e}} = \begin{array}{c}\begin{tikzpicture}
\draw[black] (0,0) -- (0,1);
\node[black] at (0,1) {$\bullet$};
\draw[red] (0.5,0) -- (0.5,1);
\node[red] at (0.5,1) {$\bullet$};
\draw[blue] (1,0) -- (1,1);
\node[blue] at (1,1) {$\bullet$};
\draw[black] (1.5,0) -- (1.5,1);
\node[black] at (1.5,1) {$\bullet$};
\end{tikzpicture}\end{array}
\qquad
l_{\un{f}} = \begin{array}{c}\begin{tikzpicture}
\draw[black] (0,0) -- (0,1) -- (0.75,1.5) -- (1.5,1) -- (1.5,0);
\draw[red] (0.5,0) -- (0.5,0.75);
\node[red] at (0.5,0.75) {$\bullet$};
\draw[blue] (1,0) -- (1,0.75);
\node[blue] at (1,0.75) {$\bullet$};
\end{tikzpicture}\end{array}
 \]
Hence $\{ l_{\un{e}}, l_{\un{f}} \}$ give a left (or right) $R$-basis for
$\Hom^{\bullet}(R, B) = B_{\id}^!$. Pairing the light 
leaf maps gives the matrix of the restriction of the intersection form
on $B_{\id}^!$:
\[
\left ( \begin{matrix} \alpha_t^2\alpha_s\alpha_u& \alpha_s\alpha_u\alpha_t \\
\alpha_s\alpha_u\alpha_t & -\alpha_t\alpha_0 
\end{matrix} \right)
\]
(where $\alpha_0 = \alpha_s + \alpha_t + \alpha_u$).
The determinant of this matrix is $\det = -\a_t^2\a_s\a_u(\a_s + \a_t)(\a_t + \a_u)$.
Inverting this matrix gives the matrix of the ($Q$-valued) form on
$B_{\id}$:
\[
E = \frac{1}{\textrm{det}}
\left ( \begin{matrix} -\a_t\a_0 & -\a_s\a_t\a_u \\
    -\a_s\a_t\a_u & \a_t^2\a_s\a_u\end{matrix} \right )
=
\left ( \begin{matrix}  \frac{\a_0}{\a_s\a_t\a_u(\a_s + \a_t)(\a_t  +\a_u)} &
  \frac{1}{\a_t(\a_s + \a_t)(\a_t  +\a_u) }\\ \frac{1}{\a_t(\a_s + \a_t)(\a_t  +\a_u) } &\frac{-1}{(\a_s + \a_t)(\a_t  +\a_u) }
 \end{matrix} \right )
\]
The determinants of the leading principal minors are:
\begin{gather*}
E_{1,1} = \frac{\a_0}{\a_s\a_t\a_u(\a_s + \a_t)(\a_t  +\a_u)} \qquad 
\det E = 1/\det
\end{gather*}
We conclude that for any regular dominant coweight $\rho^\vee \in \hg$ the
leading principal minors are $> 0$ and $<0$ respectively. Hence the
Hodge-Riemann relations are satisfied. Also $E_{1,1}$
agrees with the equivariant multiplicity at $y = \id$ in the Schubert
variety indexed by $x = tsut$ (see Theorem \ref{thm:em}).

\subsection{The second singular Schubert variety}

Let $\un{x} = sutsu$. Then
\[
BS := B(s)B(u)B(t)B(s)B(u) \cong B(sutsu) \oplus B(su)[1] \oplus B(su)[-1]
\]
One possible choice of idempotent projector to $B = B(sutsu)$ 
in $\End(BS)$ is the following morphism (with colour coding ${\color{red}s}, t
, {\color{blue}u}$ as above):
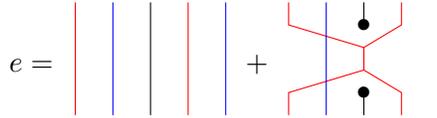
\[
e = 
\begin{array}{c}
\begin{tikzpicture}
\draw[red] (0,0) -- (0,1.5);
\draw[blue] (0.5,0) -- (0.5,1.5);
\draw[black] (1,0) -- (1,1.5);
\draw[red] (1.5,0) -- (1.5,1.5);
\draw[blue] (2,0) -- (2,1.5);
\end{tikzpicture} \end{array} + 
\begin{array}{c}
\begin{tikzpicture}
\draw[red] (0,0) -- (0,0.3) -- (1,0.6)--(1,0.9)--(0,1.2)--(0,1.5);
\draw[red] (1.5,0) -- (1.5,0.3) -- (1,0.6)--(1,0.9)--(1.5,1.2)--(1.5,1.5);
\draw[blue] (0.5,0) -- (0.5,1.5);
\node[black] at (1,0.3) {$\bullet$};
\node[black] at (1,1.2) {$\bullet$};
\draw[black] (1,0) -- (1,0.3);
\draw[black] (1,1.2) -- (1,1.5);
\draw[blue] (2,0) -- (2,1.5);
\end{tikzpicture} \end{array} 
\]
There are four subexpressions of $\un{x}$ for $x = \id$: $00000$,
$10010$, $01001$ and $11011$ of defects $5, 3, 3$ and $1$
respectively. The light leaf morphisms corresponding to the
subexpressions $10010$ and $11011$ give zero when composed with the
idempotent $e$, and the light leaf morphisms $l_5$ and $l_3$
corresponding to $00000$ and $01001$ give a basis for $B_\id^!$ after
composition with $e$. 
The matrix of the intersection form on $B_\id^!$ is given by
\[
\left ( \begin{matrix} \a_s\a_t\a_u^2(\a_s+\a_t) &
    \a_s\a_t\a_u(\a_s+\a_t) \\ \a_s\a_t\a_u(\a_s+\a_t) &
    -a_s\a_u(\a_s+2\a_t+\a_u) \end{matrix} \right )
\]
with determinant $-\a_s^2\a_t\a_u^2(\a_s+\a_t)(\a_t+\a_u)\a_0$.
Inverting this matrix gives the intersection form on $B_\id$:
\[
E = \left ( \begin{matrix} \frac{\a_s+2\a_t+\a_u}{\a_s\a_t\a_u(\a_s+\a_t)(\a_t+\a_u)\a_0}&
    \frac{1}{\a_s\a_u(\a_t+\a_u)\a_0} \\  \frac{1}{\a_s\a_u(\a_t+\a_u)\a_0}&
    \frac{-1}{\a_s(\a_t+\a_u)\a_0} \end{matrix} \right )
\]
Again we see that $E_{1,1}$ agrees with the equivariant multiplicity and
that the two leading principal minors have signatures $1$ and $0$ under
any specialisation determined by a regular dominant coweight. Hence
the Hodge-Riemann bilinear relations are satisfied.

This example has an interesting feature not seen in the previous
case. Because the numerator of $E_{1,1}$ is not a product of roots
there exist regular $\g^\vee \in \hg$ such that the evaluation of
$E_{1,1}$ at $\g^\vee$ gives zero (i.e. $\langle \a_s+2\a_t+\a_u,
\g^\vee \rangle = 0$). For such $\g^\vee$
local hard Lefschetz fails. (All that matters in this example is
$E_{1,1}$. Even if you didn't follow the
above calculation, one can calculate $E_{1,1}$ easily using the
nil Hecke ring and Theorem \ref{thm:em}.)

As explained in the
introduction, via the work of Soergel and K\"ubel this example implies that the Jantzen
filtration behaves differently for a choice of deformation direction
corresponding to $\g^\vee$. We analyse this directly. (Actually,
the example we will consider corresponds to the local intersection
form of $B$ at $y = su$, not $y = \id$ (this gives a weight space of
more manageable dimension). However the behaviour is very
similar to the above.) 

\subsection{Examples of the Jantzen filtration} 
We keep the notation above, except that now we work over $\CM$. Let
$\rho = \frac{1}{2}(3\alpha_s + 4 \alpha_t + 3 \alpha_s)$ denote
the half sum of the positive roots and let $x \cdot \lambda :=
x(\lambda + \rho) - \rho$ denote the dot action of $W$ on
$\hg^*$. We work in $\OC_0$, the principal block of category 
$\OC$ (i.e. all modules are $\mathfrak{g}$-finitely generated,
$\bg$-integrable, $\hg$-semisimple and have the same central character
as the trivial representation). We denote the Verma and simple module
of highest weight $x \cdot 0$ by $\Delta(x)$ and $L(x)$.

Motivated by the previous section we consider the Verma module
$\Delta(su)$ of highest weight $su \cdot 0 =-\alpha_s - \alpha_u$ and its weight space at
\[
\l = sutsu \cdot 0 = - 3\alpha_s - 3\alpha_t - 3\alpha_u.
\]
The dimension of the $\l$-weight space is the number of Kostant
partitions of $-(\alpha_s + \alpha_u) - \l = 2\alpha_s + 3\alpha_t +
2\alpha_u$ which is 13.

Let $\Delta := U(\mathfrak{g}) \otimes_{U(\mathfrak{b})}
S(\mathfrak{h})$ denote the universal Verma
module (of highest weight $univ$). Computing the Shapovalov form on
the weight space $univ - \nu$ with $\nu = su \cdot 0 - \l$ gives a $13
\times 13$ matrix of polynomials in $S(\hg)$. One can compute this
matrix of polynomials via computer (I used magma).  Specialising via a highest
weight $\mu : S(\hg) \to \CM$ gives the Shapovalov form on the weight
space $\mu - \nu$.

If we choose to deform via $\rho$ we get Jantzen filtration layers of
dimensions:
\[
7, 3, 2,1.
\]
Now choose $\gamma \in \hg^*$ such that $\gamma$ does not vanish
on any coroot but $\gamma(\alpha_s^\vee + 2 \alpha_t^\vee +
\alpha_u^\vee) = 0$. 
With this choice of deformation direction the Jantzen filtration layers have dimensions:
\[
7,2,4,0.
\]

By Kazhdan-Lusztig theory we have (in the Grothendieck group of $\OC_0$):
\begin{align*}
\Delta(su) = L(su) &+ L(stu) + L(tsu) + L(sut) + L(uts) + \\ & +  L(suts) +L(stut) + L(stsu) + 
L(tuts) + 2 L(stuts) + \dots
\end{align*}
(where $\dots$ consists of terms $L(x)$ with $x > sutus$). Taking the
dimension of the $\l$-weight space we get (again by Kazhdan-Lusztig theory):
\[
13 = 7 + 0 + 0 + 2 + 0 + 1 + 1 + 0 + 0 + 2.
\]

By Kazhdan-Lusztig theory one expects the following Jantzen filtration
layers
\[
\Delta(su) = \begin{matrix} L(su) & 0  \\  L(stu) \oplus L(tsu)
  \oplus L(sut) \oplus
  L(uts) \oplus L(stuts) & 1 
  \\ L(suts) \oplus L(stut) \oplus L(stsu) \oplus L(tuts) \oplus
  \dots & 2 \\
  L(stuts) \oplus \dots& 3 \\ \vdots & \vdots\end{matrix}
\]
where we have omitted any terms $L(x)$ with $x > sutus$.

Taking dimensions of weight spaces gives:
\[
13 = \begin{matrix} 7  \\  0 + 0 + 2 + 0 + 1 
  \\ 1 + 1 + 0 + 0\\
 1 \end{matrix}
\]
This matches the above calculation of filtration layers $7,3,2,1$.

One sees what happens when the Jantzen filtration
degenerates. The two subquotients isomorphic to $L(stuts)$ which
generically occur in degrees $1$ and $3$ ``slide together'' into
degree 2 so that
$\gr_2$ is no longer semi-simple.

\newpage
\section{List of notation} \label{not}

The most important cast members, in order of appearance:

\vspace{0.2cm}
\begin{tabular}[c]{ll}
$[m]$ & the shift of grading functor, \S\ref{grad} \\
$\deg_{\le i}$ & a term in the degree filtration, \S\ref{sec:deg} \\
$(W,S)$ & the fixed Coxeter system, \S\ref{sec:WS} \\
$\ell,\le$  & the length function and Bruhat order, \S\ref{sec:WS} \\
$\un{x},\un{u}$  & an expression, a subexpression, \S\ref{sec:WS}\\
$\hg$ & the reflection faithful representation of $\hg$, \S\ref{h} \\
$\rho, \rho^\vee$ & fixed dominant regular elements of $\hg^*$ and $\hg$,
\S\ref{h} \\
$R, Q$ & the regular functions on $\hg$ and its localisation at
$\Phi$, \S\ref{sec:pos} \\
$\partial_s$ & a divided difference operator, \ref{sec:pos} \\
$A, K$ & the rings $\RM[z]$ and $\RM[z^{\pm 1}]$, \S\ref{sec:pos} \\
$\sigma$ & the homomorphism $R \to A$ determined by $\rho^\vee$,
\S\ref{sec:pos} \\
$P^{-d}$ & primitive subspaces, \S\ref{def:hr}\\
$e_{x,y}$ & an equivariant multiplicity, \S\ref{sec:nilhecke} \\
$M, M_0, M_\infty$ & a $\PM^1$-sheaf and its stalks, \S\ref{sec:P1}\\
$Z, Z^2$ & the structure algebra and its degree two elements, \S\ref{sec:P1} \\
$Z^2_{\amp}$ & the ample cone in $Z^2$, \S\ref{sec:P1}. \\
$B(\un{w}), B(y)$ & a Bott-Samelson
(resp. indecomposable) Soergel bimodule, \S\ref{sbim} \\
$c_\id, c_s$ & elements in Soergel bimodules $B(s)$, \S\ref{sbim} \\
$m, \mu$ & ``dot'' maps $B(s) \to R[1]$ and $R \to B(s)[1]$,
\S\ref{sbim} \\
$B_x, B_x^!$ & the stalk and costalk of a Soergel bimodule,
\S\ref{sec:stalkcostalk} \\
$i_x$ & the inclusion $B_x^! \to B_x$, \S\ref{sec:stalkcostalk} \\
$\Gr_x$ & the (twisted) graph of $x$ inside $\hg \times \hg$,
\S\ref{sec:stalkcostalk}\\
$M(B,x,xs)$ & the $\PM^1$-sheaf associated to $B$ and $x < xs$,
\S\ref{sec:SP1}\\
\end{tabular}

\def\cprime{$'$}

\end{document}